\documentclass[12pt,reqno]{amsart}
\usepackage{amsmath,amsxtra,amssymb,amsthm,amsfonts,eufrak,bm}
\usepackage{hyperref}
\usepackage{cleveref}
\usepackage{cite}
\usepackage{graphicx}
\usepackage{bbm}
\usepackage{xcolor}
\usepackage[toc,page]{appendix}
\usepackage{subfigure}
\usepackage{mathtools}
\usepackage[normalem]{ulem}
\usepackage{accents}
\usepackage{tikz-cd}
\usetikzlibrary{matrix,arrows,decorations.pathmorphing}
\usepackage{comment}

\makeatletter
\newcommand*{\rom}[1]{\expandafter\@slowromancap\romannumeral #1@}
\makeatother

\newcommand{\kl}{\pl \le \pl}

\newcommand{\Z}{{\mathbb Z}}

\newcommand{\ten}{\otimes}

\newcommand{\pl}{\hspace{.1cm}}

\newcommand{\ran}{\rangle}
\newcommand{\lan}{\langle}
\newcommand{\al}{\alpha}

\newcommand{\la}{\lambda}
\newcommand{\eps}{\varepsilon}

\newcommand{\id}{\iota_{\infty,2}^n}

\newcommand{\E}{{\mathcal E}}

\newcommand{\M}{{\mathcal M}}
\newcommand{\bM}{{\mathbb M}}

\newcommand{\cL}{\mathcal{L}}

\newcommand{\I}{{ \bf 1}}

\newcommand{\norm}[2]{\parallel \! #1 \! \parallel_{#2}}

\newtheorem{lemma}{Lemma}[section]
\newtheorem{prop}[lemma]{Proposition}
\newtheorem{theorem}[lemma]{Theorem}
\newtheorem{cor}[lemma]{Corollary}

\newtheorem{rem}[lemma]{Remark}

\newtheorem*{acknowledgement}{Acknowledgement}

\theoremstyle{definition}
\newtheorem{definition}{Definition}[section]

\theoremstyle{plain}

\newtheorem*{theorem*}{Theorem}

\newcommand{\re}{\begin{rem}\rm}
\newcommand{\mar}{\end{rem}}
\newcommand{\bra}[1]{\langle{#1}|}
\newcommand{\ket}[1]{|{#1}\rangle}

\newcommand{\qd}{\end{proof}\vspace{0.5ex}}
\newcommand{\prf}{\begin{proof}[\bf Proof:]}

\newcommand{\xspace}{\hbox{\kern-2.5pt}}

\renewcommand{\id}{\operatorname{id}}

\newcommand{\dom}{\operatorname{dom}}

\newcommand{\Mz}{{\mathbb{M}}}

\renewcommand{\div}{\text{div}}

\newcommand{\Bf}{{f}}

\newcommand{\Bg}{{ g}}

\DeclareMathOperator{\tr}{ tr}
\DeclareMathOperator{\Ent}{ Ent}
\DeclareMathOperator{\LSI}{ LSI}
\DeclareMathOperator{\MLSI}{ MLSI}

\renewcommand{\id} {\operatorname{id}}

\allowdisplaybreaks

\newcommand{\li}[1]{{\color{red} Li says: #1}}
\newcommand{\masha}[1]{{\color{blue} Masha says: #1}}

\begin{document}

\title[CMLSI for sub-Laplacian on $\operatorname{SU}\left( 2 \right)$]{Matrix-valued modified Logarithmic Sobolev inequality for sub-Laplacian on $\operatorname{SU}\left( 2 \right)$}
%\title{Curvature approach for CLSI}
%\author{Li Gao}

\author[L.~Gao]{Li Gao}
\address{Department of Mathematics\\
 University of Houston\\
  Houston, TX 77204-3008,  U.S.A.}
  \email[Li Gao]{lgao12@uh.edu}

\author[M.~Gordina]{Maria Gordina{$^{\dag}$}}
%\thanks{\footnotemark {$\ddag$} Research was supported in part by the Simons Fellowship.}
\thanks{\footnotemark {$\dag$} Research was supported in part by NSF Grant DMS-1712427.}
\address{ Department of Mathematics\\
University of Connecticut\\
Storrs, CT 06269,  U.S.A.}
\email[Maria Gordina]{maria.gordina@uconn.edu}



%58J65: Diffusion processes and stochastic analysis on manifolds
%53C17: Sub-Riemannian geometry

%\date{\today \ \emph{File:\jobname{.tex}}}

\begin{abstract}
We prove that the canonical sub-Laplacian on $SU(2)$ admits a uniform modified log-Sobolev inequality for all its matrix-valued functions, independent of the matrix dimension. This is the first example of sub-Laplacian that a matrix-valued modified log-Sobolev inequality has been obtained. We also show that on Lie groups, the heat kernel measure $p_t$ at time $t$ admits matrix-valued modified log-Sobolev constants of order $O(t^{-1})$.
\end{abstract}
\maketitle

\section{Introduction}
Since the seminal works \cite{Gross1975c, Gross1975a} of L.~Gross, log-Sobolev inequalities have been intensively studied and found rich connections to analysis, geometry, and probability (see \cite{Ledoux,Gross14} for surveys). In recent decades, log-Sobolev inequalities has been studied for quantum systems on noncommutative spaces,  and have attracted a lot of attention in quantum information theory and quantum many-body system, e.g. \cite{OlkiewiczZegarlinski1999, CarboneSasso2008, KastoryanoTemme2013, CRF, GaoRouze2021, BardetCapelGaoLuciaPerez-GarciaRouze2021}. Motivated by the quantum information theory, we study modified log-Sobolev inequalities for matrix-valued functions for semigroups generated by sub-Laplacians, which has direct application to quantum Markov semigroups on matrix algebras.

Recall that a classical Markov semigroup $P_t=e^{t\cL}: L^{\infty}(\Omega,\mu)\to L^{\infty}(\Omega,\mu)$ on a probability space $(\Omega, \mu)$ is a semigroup of conservative (unital) positivity preserving maps. We say that $P_t$ satisfies a \emph{logarithmic Sobolev inequality $\operatorname{LSI} \left(\lambda\right)$} if for some $\la>0$
\begin{align}\label{eq:LSI1}
\int_{\Omega} f^2 \ln f^2 d\mu - \left( \int_{\Omega} f^2d\mu \right)  \ln  \left( \int_{\Omega} f^2d\mu \right) \leqslant  -\frac{2}{\lambda}\int_{\Omega} (\cL f)f d\mu
\end{align}
holds for all real-valued functions $f \in L^{2}(\Omega,\mu)$ in the domain of the generator $\cL$.  A different version of \eqref{eq:LSI1}, called a \emph{modified log-Sobolev inequality $\operatorname{MLSI} \left(\lambda\right)$}, states that for all positive functions $g\geqslant  0$

\begin{align}\label{eq:MLSI1}
\int_{\Omega} g\ln  g d\mu - \left( \int_{\Omega} g d\mu\right)  \ln   \left( \int_{\Omega} g d\mu \right) \leqslant -\frac{1}{2\lambda} \int_{\Omega} \left( \cL g \right)\ln g d\mu.
\end{align}
\begin{comment}MLSI is equivalent to the entropy decay
\begin{align}\label{eq:decay}\operatorname{Ent}(P_tg)\leqslant e^{-2\lambda t}\operatorname{Ent}(g)\pl,\end{align}
where $\operatorname{Ent}(f)=\int f\ln fd\mu$ is the entropy functional.
\end{comment}
Recently much progress has been made on extending the MLSI \eqref{eq:MLSI1} to matrix-valued functions. For instance,  H.~Li, M.~Junge and N.~LaRacuente in \cite{JLLR} proved that if a compact Riemannian manifold $\left( M, g \right)$ has the Ricci curvature bounded from below by a positive scalar $\lambda$, then for all $n \geqslant 1$ and all smooth $n\times n$-matrix-valued density functions $f: M\to \Mz_n$
\begin{align}\label{eq:CMLSI1}
\int_{M} \tr(f\ln f) d\mu -\tr\left(\mathbb{E}_{\mu} f \ln (\mathbb{E}_{\mu}f)\right) \leqslant -\frac{1}{2\lambda} \int_{M} \tr\left( (\id_{M_n}\ten \Delta)f\ln f \right)d\mu,
\end{align}
where $\Delta$ is the Laplace-Beltrami operator, $\tr$ is the standard matrix trace and $\mathbb{E}_{\mu} f=\int_{}f(x)d\mu(x)\in \Mz_n$ is the matrix-valued mean with respect to the volume form $d\mu$. This extends the well-known Bakry-{\^{'}E}mery theorem to matrix-valued functions. Equation \eqref{eq:CMLSI1} is called a \emph{complete modified log-Sobolev inequality} $\operatorname{CMLSI}\left( \lambda \right)$, as it gives a uniform MLSI constant for all matrix-valued functions independent of a matrix size. Later M.~Brannan, L.~Gao and M.~Junge in \cite{BrannanGaoJunge2022} proved that the Ricci curvature bounded from below by $\lambda$ combined with a $L^{\infty}$-mixing time implies an CMLSI. In particular, the heat semigroup on a compact Riemannian manifold always satisfies $\operatorname{CMLSI}\left( \lambda \right)$. In the discrete setting, it was proved in \cite{JLLR} that a CMLSI holds for all symmetric Markov semigroups whose generators are graph Laplacians on finite weighted undirected graphs. These results for matrix-valued functions have direct applications to quantum Markov semigroups studied in quantum information theory and quantum many-body systems, e.g. \cite[Section 6 \& 7]{gao2020fisher}.

Despite the progress for the Laplace-Beltrami operators on Riemannian manifolds and graph Laplacians, the CMLSI for sub-Laplacians is largely not explored. Suppose $M$ is an $n$-dimensional smooth manifold, and $\mathcal{D}$ is a sub-bundle of $TM$ equipped with a metric $g$ and with $\operatorname{dim}\mathcal{D}=k \leqslant n$. Recall that a second order differential operator $L$ defined on $C^{\infty}\left( M \right)$ is called a \emph{sub-Laplacian} if for every $x \in M$ there is a neighborhood $U$ of $x$ and a collection of smooth vector fields $\left\{ X_{0}, X_{1}, ..., X_{k}\right\}$ defined on $U$ such that $\left\{ X_{1}, ..., X_{k}\right\}$ are orthonormal with respect to the sub-Riemannian metric $g$ and
\[
\cL=-\sum_{i=1}^{k}X_i^{\ast}X_i+X_{0}.
\]
Here the adjoint $X_i^{\ast}$ is taken with respect to a smooth probability measure $\mu$. We assume that the vector fields $\left\{ X_{0}, X_{1}, ..., X_{k}\right\}$ satisfy H\"ormander's condition, that is, that the Lie algebra generated by these vector fields spans the whole tangent space $T_xM$ at any $x \in M$. Moreover, we will assume that $\left\{X_{1}, ..., X_{k}\right\}$ satisfy such a condition, and usually this is called a strong H\"ormander condition.  Then by the classical theorem of L.~H\"{o}rmander in \cite{Hormander1967a} the operator $L$ is hypoelliptic. In the case $X_0=0$, $\cL$ is a generator of a Markov semigroup on $L^{\infty}(M)$ that is symmetric to the measure $\mu$.

The operators satisfying a strong H\"{o}rmander condition is a central topic in the study of sub-Riemannian geometry. The literature on such geometry is vast, starting with \cite{Strichartz1986a, Strichartz1986aCorrections} and then covered in \cite{MontgomeryBook2002}. The new direction of introducing curvatures and applications to geometric analysis on sub-Riemannian manifolds has started with \cite{BaudoinGarofalo2017}, while a more analytic and probabilistic description of such sub-Laplacians has been given in \cite{GordinaLaetsch2016, GordinaLaetsch2017}.

More relevant to our results is \cite{LugiewiczZegarlinski2007}, where P.~{\L}ugiewicz and B.~Zegarli{\'n}ski proved that on compact manifolds, MLSI hold for sub-Laplacians $\cL=-\sum_{i}X_i^*X_i$ satisfying the strong H\"{o}rmander condition. Nevertheless, their method relies on the Rothaus lemma \cite{Rothaus1985} which fails in matrix-valued cases as pointed out in \cite[Section 7.5]{gao2020fisher}. Compared to the result in \cite{BrannanGaoJunge2022} in the Riemannian setting, the difficulty in the sub-Riemannian case stems from the lack of a Ricci curvature bound. Informally, at points of degeneracy of $L$, the Ricci tensor is not well-defined and might be interpreted as being $-\infty$ in some directions. While there has been a number of results on generalized notions of curvature on some classes of sub-Riemannian manifolds (see e.g. \cite{BaudoinBonnefontGarofalo2014, BaudoinGarofalo2017, BaudoinKimWang2016}), we are not relying on those in this paper.

In the current paper, we overcome this issue by relying on the  gradient estimate \eqref{eq:GE4} studied by \cite{DriverMelcher2005, Melcher2008,li2006estimation, BakryBaudoinBonnefontChafai2008, BaudoinBonnefont2009} in the sub-Riemannian setting. We denote by $\nabla=(X_1,\cdots, X_k)$ the horizontal gradient operator associated to $L$.

\begin{theorem}[Theorem \ref{thm:main}]
Let $(M, g)$ be a Riemannian manifold without boundary, and let $\mu$ be a smooth measure on $M$. Suppose  $\cL=-\sum_{i=1}^kX_i^*X_i$ is a sub-Laplacian, where $\left\{ X_{1}, ..., X_{k} \right\}$ satisfy H\"omander's condition, and denote by $P_t=e^{t\cL}$ the symmetric Markov semigroup generated by $\cL$. Suppose
\begin{enumerate}
\item[i)] there exists some $t_0>0$ and a positive measurable function $C:[0,t_0)\to \mathbb{R}_+$ such that
for all $f\in C^\infty(M)$ and $0\leqslant t< t_0$,
\begin{align}\label{eq:GE4}
|\nabla P_t f|^2\leqslant C(t)P_t (|\nabla f|^2),
\end{align}
where $|\nabla f|^2=\sum_{i}|X_if|^2$.
\item[ii)] $\displaystyle \int_{0}^{t_0} C(s)ds<\infty$ is finite.
\end{enumerate}
Then $P_t$ satisfies $\operatorname{CMLSI}\left( \lambda \right)$ for some positive $\lambda$.
\end{theorem}

We note that if $C(t)=e^{-2\lambda t}$, \eqref{eq:GE4} is exactly Bakry-{\'E}mery's $(\lambda,\infty)$ curvature-dimension inequality. The results by B.~Driver-T.~Melcher in \cite{DriverMelcher2005}, T.~Melcher in \cite{Melcher2008} and F.~Baudoin-M.~Bonnefont in \cite{BaudoinBonnefont2009} imply that \eqref{eq:GE4} holds on $\operatorname{SU}\left( 2 \right)$, and on stratified Lie groups including Heisenberg group. In particular, on stratified Lie groups  $C(t)\equiv C$ is a constant function. The main ingredient in \cite{DriverMelcher2005, Melcher2008} is Malliavin's calculus, while \cite{BaudoinBonnefont2009} relied on the spectral decomposition of the heat kernel. As a corollary of our main result we obtain an CMLSI for the canonical sub-Laplacian on $\operatorname{SU}\left( 2 \right)$.

Recall that the Lie algebra $\mathfrak{su}(2)$ is spanned by  skew-Hermitian traceless matrices
\[
X=\left[\begin{array}{cc}
0& 1
\\
-1& 0
\end{array}
\right],
Y=\left[\begin{array}{cc} 0& i\\ i& 0
\end{array}\right],
Z=\left[\begin{array}{cc} i& 0\\ 0& -i
\end{array}\right],
\]
whose Lie brackets satisfy $[X,Y]=2Z, [Y,Z]=2X$, and $[Z,X]=2Y$. In particular, it implies that$\{X,Y\}$ satisfies H\"ormander's condition.
\begin{cor}\label{cor:su2}
The sub-Laplacian $\cL=X^2+Y^2$ on $\operatorname{SU}\left( 2 \right)$ satisfies a $\operatorname{CMLSI}\left( \lambda \right)$ for some positive $\la$.
\end{cor}
The above result is the first ever example of CMLSI obtained for sub-Laplacians.
It also has direct implications for the CMLSI constant for quantum Markov semigroups. Quantum Markov semigroups are noncommutative generalizations of Markov semigroups, where the underlying function spaces are replaced by matrix algebras or operator algebras. Mathematically, they are continuous semigroup of completely positive trace preserving map.
Quantum Markov semigroups satisfy what is known as a GKSL equation or master equation as in \cite{ChruscinskiPascazio2017}, and they model the Markovian time evolution of an open quantum system. There is a wide interest in quantum information theory on the convergence rate of quantum Markov semigroup.

Thanks to \cite{GoriniKossakowskiSudarshan1976,LindbladG1976}, it is known that if a quantum Markov semigroup on matrix algebra $\Mz_m$ is symmetric with respect to the matrix trace, then it admits the following Lindbladian form,
\[
S_t=e^{\cL t}:\Mz_m\to \Mz_m\ ,\ \cL(\rho)=\sum_{j=1}[a_j,[a_j,\rho]],
\]
where $a_j$ are some self-adjoint matrices. As an application of Corollary \ref{cor:su2}, we obtain a uniform lower bound of CMLSI constants for quantum Markov semigroups induced by the sub-Laplacian $\cL=X^2+Y^2$ via Lie algebra representations of $\mathfrak{su}(2)$.

\begin{cor}
There exists a positive constant $\lambda$ such that for any Lie algebra homomorphism $\phi:\mathfrak{su}(2)\to i(\Mz_m)_{s.a.}$ into a matrix algebra $\Mz_m$, the quantum Markov semigroup
\begin{align*}
& S_t=e^{\cL t}:\Mz_m\to \Mz_m,
\\
&\cL(\rho)=[\phi_m(X),[\phi_m(X),\rho]]+[\phi_m(Y),[\phi_m(Y),\rho]]
\end{align*}
satisfies $\operatorname{CMLSI}\left( \lambda \right)$.
\end{cor}

Motivated by the application of $L^{2}$-gradient estimate \eqref{eq:GE4} to CMLSI, we also investigate whether $L^{1}$-gradient estimate
\begin{align}|\nabla P_t(f)|\leqslant C(t)P_t|\nabla (f)| \label{eq:p1}\end{align}
can be used to derive CMLSI. Let $p_t$ be the heat kernel measure of the semigroup map $P_t$ at time $t$.
For scalar value case, it is known that \eqref{eq:p1} implies the heat kernel measure $p_t$ admits LSI constant given by
\[\kappa^{-1}=\int_{0}^t C(s)^2ds\]
This corresponds to the associated Ornstein–Uhlenbeck semigroup $O_s=e^{L_t s}$ where $L_t=-\nabla^*\nabla$ is the symmetric generator on $L^{2}(dp_t)$. Our last result shows this approach works for CMLSI on Lie groups.

\begin{theorem}[c.f. Theorem \ref{thm:OU}] Let $G$ be a locally compact Lie group equipped with a left invariant metric. Let $H_t=e^{\Delta t}$ be the heat semigroup and denote $p_t$ as the heat kernel measure. Then, the Ornstein–Uhlenbeck semigroup $O_s=e^{L_{t}s}$ satisfies $\frac{1}{4t}$-\text{CMLSI}, where $L_t=\nabla^*\nabla$ is the generator to the measure $p_t$.
\end{theorem}

The rest of the paper is organized as follows. In Section~\ref{s.Prelim}, we review some preliminaries on CMLSI and sub-Laplacians. In Section~\ref{sec:fisher} we prove our main Theorem \ref{thm:main} and apply it to $\operatorname{SU}\left( 2 \right)$. Section \ref{sec:ou} discussed CMLSI for heat kernel measure on Lie groups.
We end the paper with some discussion connections of our results to quantum Markov semigroups and list some open questions.

\subsection*{Notation}
Throughout the paper, we denote by $(\Omega,\mu)$ a measure space equipped with a probability measure, and by $L^{p}(\Omega)$ the corresponding $L^{p}$-space of complex-valued functions for $1 \leqslant p\leqslant \infty$. Then $\norm{f}{p}$ is the standard $L^{p}$-norm and $\lan f, g\ran=\int_{\Omega} \overline{f}g d\mu$ is the $L^{2}$ inner product. By $\mathbb{M}_n$ we denote the space of $n\times n$ complex matrices and by $\tr$ the standard matrix trace. The identity elements, which is the constant function $1$ for $n=1$ and the identity operator in $\mathbb{M}_n$, is denoted by $\mathbf{1}$, and the identity map between spaces is denoted as $\id$.

\begin{acknowledgement}
{\rm L.G. is grateful to Marius Junge and Melchior Wirth for helpful discussions.}
\end{acknowledgement}

\section{Preliminaries}\label{s.Prelim}
\subsection{Logarithmic Sobolev inequalities}
We first recall logarithmic Sobolev inequalities for Markov semigroups and its matrix-valued extension. Let $(\Omega,\mu)$ be a measure space equipped with a probability measure $\mu$. We say that $P:L^{\infty}(\Omega)\to L^{\infty}(\Omega)$ is a \emph{Markov map} if
\begin{enumerate}
\item[i)]  $P_t(1)=1$ (mass conservation);
\item[ii)] $P_t(f)\geqslant  0$ if $f\geqslant  0$ (positivity preserving);
\end{enumerate}
A \emph{Markov semigroup} $(P_t)_{t\geqslant  0}: L^{\infty}(\Omega)\to L^{\infty}(\Omega)$ is a family of Markov maps satisfying
\begin{enumerate}
\item[i)] $P_{0}=\id$ and $P_s\circ P_t=P_{s+t}$ for $s,t\geqslant  0$ (semigroup property);
\item[ii)] For every $f \in L^{2}$, $P_{t}f$ converges to $f$ in $L^{2}$) as $t \to 0$ (continuity property).
%\textcolor{red}{Li: what is the standard continuity assumption in classical case?}
\end{enumerate}
The generator of $P_t$ is given by
\[
\cL f= \lim_{t\to 0}\frac{P_tf-f}{t}\pl, \pl P_t=e^{\cL t}\pl,
\]
with the domain $\dom(\cL)$ being the space of functions such that the above limit exists.
Throughout the paper, we consider the semigroups which are symmetric with respect to a unique invariant measure $\mu$, i.e.
$\lan P_t(f),g\ran_{L^{2}(\mu)}=\lan f,P_t(g)\ran_{L^{2}(\mu)}$ for any $t\geqslant  0$. Namely, each $P_t$ is a symmetric operator on $L^{2}(\Omega)$ and hence $\int_{\Omega} P_t(f)d\mu= \int_{\Omega} f d\mu$ . In this case, $\cL$ is a negative operator on $L^{2}(\Omega)$ and $P_t$ is equivalently determined by its Dirichlet form
\[
\E (f,g)=\int_{\Omega} \bar{f}\cL(g)d\mu
\]
whose domain is $\dom(\E)=\dom((-\cL)^{1/2})$.

\begin{definition}
For $\la>0$, a Markov semigroup $P_t=e^{t\cL}$ is said to satisfy
\begin{enumerate}
\item[i)] \emph{$\lambda$-Poincar\'e inequality} $\operatorname{PI}\left( \lambda \right)$ if for any $f\in \dom(\E)$
\begin{align}\label{eq:PI}
\lambda \norm{f-\mathbb{E}_{\mu} f}{2} \leqslant 2 \E(f,f),
\end{align}
where $\mathbb{E}_{\mu} f=\int_{\Omega} f d\mu$.
\item[ii)] \emph{$\lambda$-logarithmic Sobolev inequality} $\operatorname{LSI}\left( \lambda \right)$ if for any $f\in \dom(\E)$,
\begin{align}\label{eq:LSI}
 \int_{\Omega} |f|^2 \ln |f|^2 d\mu- \left(\int_{\Omega} |f|^2d\mu\right)\ln \left(\int_{\Omega} |f|^2d\mu\right) \leqslant \frac{2}{\lambda} \E(f,f);
\end{align}
\item[iii)] \emph{$\lambda$-modified logarithmic Sobolev inequality} $\operatorname{MLSI}\left( \lambda \right)$ if for any positive $g\in \dom(\cL)$,
\begin{align}\label{eq:MLSI}
\int_{\Omega} g \ln g \ d\mu- \left(\int_{\Omega} g \ d\mu\right)\ln \left(\int_{\Omega} g\ d\mu\right) \leqslant- \frac{1}{2\lambda}\int_{\Omega} \left( \cL g \right)\ln g d\mu \pl.
\end{align}
\end{enumerate}
\end{definition}
The $\lambda$-Poincar\'e inequality is equivalent to that the generator $(-\cL)$ has a spectral gap $\lambda$. The LSI is an equivalent formulation of hypercontractivity stating that
\[
\Vert P_t \Vert_{L^{2}(\Omega)\to L^{p}(\Omega)} \leqslant 1\pl \text{ for } \pl t\leqslant 1+e^{2\lambda t}.
\]
The MLSI is known to describe the \emph{entropy decay}
\begin{align}\label{eq:decay}
\operatorname{Ent}(P_tg)\leqslant e^{-2\lambda t}\operatorname{Ent}(g),
\end{align}
where $\displaystyle \operatorname{Ent}(f)=\int_{\Omega} f\ln fd\mu$ is the \emph{entropy functional}. Note that the right side of \eqref{eq:MLSI},
\[
I(g):=-\int \  (\cL g)\ln g\  d\mu=-\left.\frac{d}{dt}\right|_{t=0} \operatorname{Ent}(P_tg)
\]
is called the \emph{Fisher information} (also called entropy production), as it is the negative derivative of entropy along the time $t$.

The exponential entropy decay \eqref{eq:decay} in particular implies mixing time in $L^{1}$ via Pinsker's inequality ( c.f. \cite[Appendix]{Verdu2014})
\[
\frac{1}{2}\norm{f-1}{1}^2\leqslant \sqrt{\operatorname{Ent}(f)}.
\]
It is well-known (c.f. \cite[Theorem 5.2.1]{BakryGentilLedouxBook}) that $\lambda$-LSI implies $\lambda$-MLSI, and they are equivalent if the semigroup is diffusive. Recall that a \emph{carr\'e du champ} operator or the gradient form of $\mathcal{L}$ is
\[
\Gamma(f,g):=\frac{1}{2}(\cL(fg)-f(\cL g)-(\cL f)g)
\]
or weakly  defined as $\lan h, \Gamma(f,g)\ran = \frac{1}{2}(\E(hf, g)+\E(f, gh)-\E(h, fg)) $. We will often use the short notation $\Gamma(f):=\Gamma(f,f)$. $P_t$ is called \emph{diffusive} if the gradient form satisfies the following product rule that
\[ \Gamma(fg,h) = f\Gamma(g,h)+\Gamma(f,h)g\pl.\]

In this paper, we study MLSIs for matrix-valued functions. This is motivated by the study of quantum Markov semigroups in the noncommutative analysis and quantum information theory.
Both LSI and MLSI enjoy the tensorization property, e.g. \cite[Section 7.6.3]{BakryGentilLedouxBook}: if two Markov semigroups $T_t$ and $S_t$ both satisfies (M)LSI, so does their tensor product semigroup $T_t\ten S_t$. More precisely, if we denote $\la_{\LSI}$ (resp. $\la_{\MLSI}$) for optimal constant such \eqref{eq:LSI} (resp. \eqref{eq:MLSI}) holds, then tensorization property states that
\[
\la_{\LSI}(T_t\ten S_t)=\min\{\la_{\LSI}(T_t), \la_{\LSI}(S_t)\}
\]
and the similar equality holds for $\la_{\MLSI}$.

Nevertheless, in the noncommutative setting when the semigroup describes a quantum system modeled by matrix algebras or operator algebras, tensorization property fails for non-primitive semigroup (non-unique invariant states \cite{BrannanGaoJunge2022}) and is largely unknown for primitive cases (see \cite{king2014hypercontractivity,beigi2020quantum} for positive results on $2$-dimensional matrices). It turns our the above tensorization property in noncommutative setting hold for a stronger version of MLSI that is uniform for all its matrix-valued amplification $T_t\ten \id_{\bM_n}$. Here and in the following, we denote by $\mathbb{M}_n$ of $n\times n$ complex matrix algebra and $\tr$ for the standard matrix trace.
We say $f:\Omega \to \mathbb{M}_n$ is a matrix-valued density if at each $\omega\in \Omega$, $f(\omega)\geqslant 0$ is  positive semi-definite and $\int_{\Omega}\tr(f) d\mu=1$.
\begin{definition}
We say a Markov semigroup $T_t:L^{\infty}(\Omega)\to L^{\infty}(\Omega)$ satisfies a \emph{complete modified logarithmic Sobolev inequality} $\operatorname{CMLSI}\left( \lambda \right)$ with $\lambda>0$ if
\begin{align}\label{eq:CMLSI}
\int_{\Omega} \tr(f\ln f) d\mu -\tr\Big(\mathbb{E}_{\mu} f \ln (\mathbb{E}_{\mu}f)\Big)\leqslant \frac{1}{2\lambda} \int\tr\Big( (\id_{\bM_n}\ten \cL)f\ln f \Big)d\mu
\end{align}
for all $n\in \mathbb{N}^+$ and matrix-valued density $f$.
\end{definition}
Here $\mathbb{E}_{\mu}f=\int_{\Omega} fd\mu \in \mathbb{M}_n$ is the matrix-valued mean, and $f\ln f$ is interpreted as the matrix-valued function that at each point $\omega\in \Omega$, $f\ln f(\omega)=f(\omega)\ln f(\omega)$ is the pointwise functional calculus of $f(\omega)$ (similarly for $\mathbb{E}_{\mu}f \ln \mathbb{E}_{\mu}f$). In other words, $T_t$ satisfies $\operatorname{CMLSI}\left( \lambda \right)$ if for all $n\geqslant  1$, the matrix-valued semigroup $T_t\ten \id_{M_n}$ satisfies $\lambda$-MLSI.
The left-hand side in \eqref{eq:CMLSI} is the relative entropy of $f$ to its matrix-valued mean $\mathbb{E}_{\mu}f$ as a constant function
\begin{align}\label{eq:re} D(f||\mathbb{E}_{\mu} f ):=\int_{\Omega} \tr\Big(f\ln f-\mathbb{E}_{\mu} f \ln (\mathbb{E}_{\mu}f) \Big) d\mu=\int_{\Omega}D(f(\omega)||\mathbb{E}_{\mu}f)d\mu(\omega)\pl,\end{align}
which is a mixture of classical relative entropy $D(f||g)=\int_{} f\ln f-f\ln g\ d\mu$ for density function $f, g \in L^{1}(\Omega)$ and quantum relative entropy $D(\rho||\sigma)=\tr(\rho\ln\rho-\rho \ln\sigma)$ for density operators $\rho,\sigma\in \mathbb{M}_n$. The right-hand side of \eqref{eq:CMLSI} is again the Fisher information
\[
I(f):=-\frac{d}{dt}D\big(\left(\id_{\mathbb{M}_n}\ten T_t \right)f|| \mathbb{E}_{\mu} f\big)|_{t=0}=-\int\tr\big( \left(\id_{\bM_n}\ten \cL\right )f\ln f \big)d\mu \pl.
\]
Note that here $\mathbb{E}_{\mu} f$ as a constant matrix-valued function is invariant under the amplified semigroup $T_t \ten \id_{\mathbb{M}_n}$. Then \eqref{eq:CMLSI} is equivalent to the convergence of $T_t f$ to $\mathbb{E}_{\mu}f$ as an equilibrium state in terms of entropy
\[
D((T_t\ten \id_{\mathbb{M}_n}) f||\mathbb{E}_{\mu} f )\leqslant e^{-2\lambda t} D(f||\mathbb{E}_{\mu} f)\pl.
\]

\subsection{Sub-Laplacians}
Let $(M,g)$ be a $d$-dimensional Riemannian manifold without boundary and $H=\{X_i\}_{i=1}^k$ with $k\leqslant d$ be a family of vectors fields. Let $d\mu=\rho d\operatorname{vol}$ be a probability measure with smooth density $\rho$ w.r.t to the volume form. Denote $\nabla=(X_1, \cdots, X_k)$ and by $X_i^*$ the adjoint of $X_i$ on $L^{2}(M,d\mu)$.
The sub-Laplacian
\[
\cL=-\nabla^*\nabla= -\sum_{i} X_i^*X_i=\sum_{i}X_i^2+\div_{\mu}(X_i)X_i
\]
is a symmetric generator on $L^{2}(M,\mu)$. Here $\operatorname{div}_{\mu}(X)$ is the divergence of $X$ w.r.t to $\mu$ and $\cL$ depends both on the family $H$ and the measure $\mu$.
The horizontal \emph{gradient form} is
\[
\Gamma(f,g)=\sum_{i}\lan X_i f,X_ig\ran\pl,\pl \Gamma(f):=\Gamma(f,f)=\sum_{i}|X_i f|^2=|\nabla f|^2 \]
It follows from the product rule that $\Gamma$ is diffusive, and the Fisher information can be rewritten as
\[ I(f):=-\int_{\Omega} (\cL f)\ln f \ d\mu= \lan \nabla f, \nabla\ln f \ran=\int_{\Omega} \frac{|\nabla f|^2}{f}d\mu\pl,\]
where we used the chain rule $\nabla(\ln f)=f^{-1}(\nabla f)$. Throughout the paper, we will use the short notation $\cL f$ for $(\cL \ten \id_{\mathbb{M}_n})f$ for the matrix-valued function $f$, and similarly for $\nabla$ and $P_t=e^{\Delta t}$. Recall the noncommutative chain rule that for a positive operator $A$ and derivation $\delta$
 \begin{align}
 \delta(\ln A)=\int_{0}^\infty(A+s\I)^{-1}\delta(A)(A+s\I)^{-1}ds\pl,\label{eq:chain}\end{align}
 where $\I$ is the identity operator. Then for a matrix-valued density $f$,
the entropy production can be rewritten as
\[
I(f)=-\int\tr\Big(  (\cL f)\ln f \Big)=\lan \nabla f, \nabla \ln f\ran_{\tr}= \int_{0}^\infty\lan \nabla f , (f+s\I)^{-1}(\nabla f)(f+s\I)^{-1}\ran_{\tr} ds\pl.
\]
Here and in the following, we denote by $\lan \cdot, \cdot\ran_{\tr}$ the integral-trace inner product for two families of matrix-valued functions $(f_i)$ and $(g_i)$
\begin{align}\label{eq:inner}
\lan (f_i), (g_i)\ran_{\tr}=\sum_{i=1}^n\int\tr(f_i^*g_i)d\mu
\end{align}
The identity element $\I\in L^{\infty}(M,\Mz_n)$ is the constant function on $M$ of the identity matrix,  and $(f+s\I)^{-1}$ is the pointwise inverse matrix.

Throughout the paper, we will assume that the family of vector fields of $H=\{X_i\}_{i=1}^k$ satisfies  \emph{H\"{o}rmander's condition}, that is, at every point $x\in M$ the tangent space at $x$ is spanned by the iterated Lie brackets of $X_i$s
\[
T_xM=\operatorname{span}\{[X_{i_1},[X_{i_2},\cdots, [X_{i_{n-1}}, X_{i_n}]]],  1\leqslant i_1,i_2\cdots i_n\leqslant k  \}.
\]
By compactness we can assume there is a global constant $l_X$ such that for every point $x\in M$, we only need at most $l_X$th iterated Lie bracket in above expression.

It follows that $\operatorname{ker}{\cL}=\{\mathbb{C}1\}$ and $\mu$ is the unique invariant measure such that $\int P_tf \ d\mu=\int f\ d\mu  $. Moreover, by a celebrated theorem of H\"ormander, $\cL$ is hypoelliptic. Indeed, we have the following Sobolev-type inequality (see e.g. \cite[Lemma 2.1]{LugiewiczZegarlinski2007})
\begin{align}\label{eq:sobolev}
\|f\|_{q} \kl C \big( \lan \cL f,f\ran + \norm{f}{2}^2\big)^{1/2} \pl,
\end{align}
where $q=\frac{2dl_X}{dl_X-2}>2$. By Varopoulos' Theorem (see \cite[Chapter 2]{VSCC}) on the dimension of semigroups, this implies the following ultra-contractivity property for $P_{t}=e^{\mathcal{L}t}$
\begin{align}\label{eq:ultra}
\Vert P_t\Vert_{L^{1}(M,\mu)\to L^{\infty}(M,\mu)}\leqslant C^{\prime} t^{-m/2} \text{ for } 0 < t \leqslant 1\pl,
\end{align}
where $m=dl_x$. The Sobolev-type inequality \eqref{eq:sobolev} and the ultra-contractivity \eqref{eq:ultra} were used in \cite{LugiewiczZegarlinski2007} to prove that every sub-Laplacian $\cL=-\sum_{i} X_i^*X_i$ with H\"ormander condition on a compact manifold satisfies LSI (hence equivalently MLSI). Their proof relies on the Rothaus lemma
\[
 \Ent(f^2)\leqslant \Ent(\mathring{f}^2)+2\norm{\mathring{f}}{2}^2,
\]
 where $\mathring{f}=f-\mathbb{E}_{\mu} f$ is the mean zero part of $f$, which is a standard tool to improve a  defective Logarithmic Sobolev inequality to a standard one. Nevertheless, the Rothaus lemma is known to fail for matrix-valued functions \cite[Section 7.5]{gao2020fisher}, hence such an argument does not apply to the matrix-valued case.

\section{Gradient estimates and Fisher information}\label{sec:fisher}

\subsection{Complete gradient estimates}
Let $(M,g)$ be a Riemannian manifold. We consider $\cL=-\sum_{i=1}^kX_i^*X_i$ be a sub-Laplacian satisfying H\"ormander condition. Denote $P_t=e^{\cL t}$ and $\nabla(f)=(X_1f,\cdots, X_kf)$.
The key tool in our argument is the following gradient estimate
\begin{align}|\nabla  P_t f|^2\leqslant C(t)P_t(|\nabla f|^2)\pl, \pl f\in C_c^\infty(M)\label{eq:gradient}\end{align}
for some function $C(t)$. In terms of gradient form, \eqref{eq:gradient} can be rewritten as
\begin{align}\Gamma(P_t f,P_t f)\leqslant C(t)P_t \Gamma( f,f)\pl. \label{eq:gradient2}\end{align}
This is closely related to the Bakry-{\'E}mery curvature-dimension condition. It was shown in \cite{BakryEmery1985} that if the Ricci curvature tensor of $( M, g)$ has a uniform lower bound $\lambda\in \mathbb{R}$, then \eqref{eq:gradient2} is satisfied with $C(t)=e^{-2\lambda t}$. In the sub-Riemannian case, although the Ricci tensor should be interpreted as $-\infty$ at the points where $L$ is degenerate, the gradient estimate \eqref{eq:gradient2} might hold for a function $C(t)$ other than exponential function $e^{-\lambda t}$.

Our first proposition shows that the gradient estimate \eqref{eq:gradient2} automatically extends to matrix-valued functions. We denote by $C_c^\infty(M,\Mz_n)$ as compactly supported smooth $\Mz_n$-valued functions. For $f, g\in C^\infty_c(M,\Mz_n)$, $f\geqslant  g$ means that at each point $x\in M$, $f(x)-g(x)\in (\Mz_n)_+$.

\begin{prop}\label{prop:gradient}
Let $(M,g)$ be a Riemannian manifold with a sub-Laplacian $\cL=-\sum_{i=1}^kX_i^*X_i$ and let $P:L^{\infty}(M,d\mu)\to L^{\infty}(M,d\mu)$ be a Markov map preserving $\mu$. Suppose
for some constant $C$ and all \textbf{scalar-valued} functions $f\in C^\infty_c(M)$
\[
\Gamma(P f, P f)\leqslant C P\Gamma(f,f)\pl.
\]
 Then
\begin{enumerate}
\item[i)] for any family of scalar-valued functions $\{f_1,\cdots,f_n\}\subset C^\infty_c(M)$,
\[
\big[\Gamma(Pf_i,Pf_j)\big]_{i,j=1}^n\leqslant C[P\Gamma(f_i,f_j)]_{i,j=1}^n\pl.
\]
Here $[\Gamma(Pf_i,Pf_j)]_{i,j}$ and $[P\Gamma(f_i,f_j)]_{i,j}$ are viewed as elements in $C^\infty_c(M,\Mz_n)$.
\item[ii)]Let $0\leqslant s\leqslant 1$. For any matrix-valued function $\Bf ,A,B\in C^\infty_c(M,\Mz_n)$ with $A,B\geqslant  0$,
\[
\lan \nabla P \Bf, A^s (\nabla P \Bf)B^{1-s} \ran_{\tr} \leqslant C \lan\nabla \Bf,  (P^\dagger A)^s (\nabla  \Bf)(P^\dagger B)^{1-s} \ran_{\tr},
\]
where $P^\dagger$ is the adjoint map of $P$ on $L^{2}(M, d\mu)$, and $\lan \cdot,\cdot\ran_{\tr}$ is the inner product defined by \eqref{eq:inner}. In particular, for any $\Bf=(f_{ij})_{i,j=1}^n\in C^\infty_c(M,\Mz_n)$,
\[
\Gamma(P\Bf,P\Bf)\leqslant CP\Gamma(\Bf,\Bf).
\]
Here $\Gamma(\Bf,\Bf)=\big[\sum_{l=1}^n\Gamma(f_{li},f_{lj})\big]_{i,j}\in C^\infty_c(M,\Mz_n)$ and similarly for $\Gamma(P\Bf,P\Bf)$.
\end{enumerate}
\end{prop}

\begin{proof} We use the standard bra-ket notation $\ket{h}$ for a vector in $\mathbb{C}^n$ and $\bra{h}$ for the dual vector. Let $\{\ket{i}\}_{i=1}^n$ be an orthonormal basis of $\mathbb{C}^n$.
For i), we have for any vector $\ket{h}=\sum_{i}h_i\ket{i}\in \mathbb{C}^n$,
\begin{align*} \bra{h}\big[\Gamma(Pf_i,Pf_j)\big]\ket{h}
=&\sum_{i,j}\bar{h}_ih_j\Gamma(Pf_i,Pf_j)
=\Gamma(Pf_h,Pf_h)
\\\leqslant &CP\Gamma(f_h,f_h)
= C \sum_{i,j}\bar{h}_ih_jP\Gamma(f_i,f_j)
= C \bra{h} [P\Gamma(f_i,f_j) ]\ket{h},
\end{align*}
where $f_h=\sum_i h_if_i\in C^\infty(M)$ and the inequality holds pointwise for each $x\in M$.
For ii), we write $A^s=\sum_{i,j}A_{i,j}^{s}(x)\ket{i}\bra{j}$ and $B^{1-s}=\sum_{k,l}B_{k,l}^{1-s}(x)\ket{k}\bra{l}$. Note that $A_{i,j}^{s}(x)$ is the coefficient function for $A^s$ not the $s$-power of $A_{i,j}(x)$.

For $\Bf=\sum_{i,j}f_{i,j} \ket{i}\bra{j}$ we have that for each $x\in M$
\begin{align} &\tr\Big( \big(\nabla P f(x)\big)^* A^{s}(x) \big(\nabla P f(x)\big)B^{1-s}(x)\Big )\nonumber\\
=&\sum_{m}\sum_{i,j,k,l} A_{i,j}^{s}(x)B_{k,l}^{1-s}(x)  \bra{l}(X_m P f(x))^*\ket{i}\bra{j} (X_m P f(x))\ket{k}\nonumber
\\ =&\sum_{m}\sum_{i,j,k,l}A_{i,j}^{s}(x)B_{k,l}^{1-s}(x)  \overline{X_m P f_{i,l}(x)} X_m P f_{j,k}(x)\label{eq:1}
\\=&\sum_{i,j,k,l}A_{i,j}^{s}(x)B_{k,l}^{1-s}(x)  \Gamma(P f_{i,l}, P f_{j,k})(x)\nonumber
\\\leqslant &C\sum_{i,j,k,l}A_{i,j}^{s}(x)\overline{B_{l,k}^{1-s}}(x)  P\Gamma( f_{i,l},  f_{j,k})(x)\label{eq:2}
\end{align}
where the equality \eqref{eq:1} follows from the fact that the evaluation $\bra{i}\cdot\ket{j}$ is a linear functional, hence commutes with $\nabla$ and $P$. Since $B$ and $B^{1-s}$ are pointwise positive in $C^\infty(M,\Mz_n)$, we have $B_{k,l}^{1-s}(x)=\overline{B_{l,k}^{1-s}}(x)$ for every $x$. Then the inequality \eqref{eq:2} follows from the assumption that
\[
\big[\Gamma(P f_{i,l}, P f_{j,k})\big]_{il,jk}\leqslant \big[P\Gamma( f_{i,l},  f_{j,k})\big]_{il,jk}
\]
and $\big[A_{i,j}^{s}\overline{B_{l,k}^{1-s}}\big]_{il,jk}$ is a positive matrix in terms of indices $(il,jk)$. Let
\[
P(f)(x)=\int_{M} f(y)dm(x,y)
\]
be the kernel representation. The adjoint is then $P^\dagger(f)(y)=\int_{M} f(x)dm(x,y)$. Integrating \eqref{eq:2} over $M$, we have
\begin{align*}
\lan (\nabla P \Bf\big), A^{s}(\nabla P \Bf)B^{1-s}\ran_{\tr}=&
\int_M\tr\Big( \big(\nabla P \Bf(x)\big)^* A^{s}(x) \big(\nabla P \Bf(x)\big)B^{1-s}(x)\Big )d\mu(x)
\\ \leqslant & C \int_{M}\sum_{i,j,k,l}A_{i,j}^{s}(x)\overline{B_{l,k}^{1-s}}(x)  P\Gamma( f_{i,l},  f_{j,k})(x)d\mu(x)
\\  = & C\int_{M}\sum_{i,j,k,l} P^\dagger(A_{i,j}^{s}\overline{B_{l,k}^{1-s}})(y)  \Gamma( f_{i,l},  f_{j,k})(y)d\mu(y)
\\=&
C\sum_{i,j,k,l} \int_{M}\Big(\int_{M}A_{i,j}^{s}(x)B_{k,l}^{1-s}(x) dm(x,y) \Big) \Gamma( f_{i,l},  f_{j,k})(y)d\mu(y)
\\=&
 C\int_{M} \Big( \int_{M}\tr\big((\nabla \Bf(y))^* A^s(x) (\nabla \Bf(y))B^{1-s}(x)\big)dm(x,y)\Big)d\mu(y)
\\\leqslant &
C\int_{M} \tr\Big((\nabla \Bf(y))^* (P^\dagger A(y))^{1-s} (\nabla \Bf(y))(P^\dagger B(y))^{s} \Big)d\mu(y)
\\ = & C\lan\nabla \Bf,  (P^\dagger A)^s (\nabla  \Bf)(P^\dagger B)^{1-s} \ran_{\tr},
\end{align*}
where the last inequality used the Lieb concavity in \cite{Lieb1973} for
the function
\[
(A,B)\mapsto \tr(K^*A^{s}KB^{1-s})
\]
is joint concave for $(A,B)\in \Mz_m$. The last assertion follows from choosing $s=0$ or $s=1$.

\end{proof}

\begin{lemma}\label{lemma:fisher}
Under the assumptions of Proposition~\ref{prop:gradient}, if in addition $[P,\cL]=0$, then for any matrix-valued density $\rho$
\begin{align} \label{eq:fisher}
I(P^\dagger\rho)\leqslant CI(\rho)\pl.
\end{align}
\end{lemma}
\begin{proof} Recall that for a matrix-valued function $\Bf \in C^\infty(M,\Mz_n)$
\[
\nabla \Bf=(X_i\Bf)_{i=1}^k\in \bigoplus_{i=1}^k L^{\infty}(M,\Mz_n)\cong L^{\infty}(M,\oplus_{i=1}^k \Mz_n) \pl,
\]
where the inner product is defined as
\[
\lan\nabla \Bf,\nabla \Bg\ran_{\tr}=\sum_{i=1}^k\int_{M}\tr\Big((X_i \Bf)^*(x)(X_i\Bg)(x)\Big)d\mu(x)\pl.
\]
Given a matrix-valued density function $\rho$, define the operator
\begin{align*}
& M_{\rho}: L^{2}(M,\oplus_{i=1}^k \Mz_n)\longrightarrow L^{2}(M,\oplus_{i=1}^k \Mz_n)\pl,
\\
& M_{\rho}\left( (f_i)_{i=1}^k \right):=\left( \int_{0}^1\rho^{s}f_i\rho^{1-s}ds \right)_{i=1}^k,
\\
& K_{\rho}=\nabla^* M_{\rho}\nabla: L^{2}(M, \Mz_n)\to L^{2}(M, \Mz_n).
\end{align*}
For simplicity, we assume that $\rho\in C^\infty(M)$ satisfies $\mu_2 \bf 1\leqslant \rho\leqslant \mu_1 \bf 1$ for some $0<\mu_1<\mu_2$. Then $\mu_1\id \leqslant M_{\rho}\leqslant \mu_2\id$ is a bounded positive operator on $L^{2}(M, \oplus_{i=1}^k \Mz_n)$. In this case, $K_{\rho}$ has the same domain and kernel as $L=-\nabla^*\nabla$. Then Proposition \ref{prop:gradient} implies that
\begin{equation}\label{e.19}
P^\dagger K_{\rho}P\leqslant C K_{P^\dagger \rho}
\end{equation}
as a positive operator on $L^{2}(M, \Mz_n)$. Note that here $K_{P^\dagger \rho}$ is a positive operator because the adjoint $P^\dagger$ is also a positive map. Thus we can define the inverse $K_{\rho}^{-1}$ as a densely defined operator on the support $\ker(L)^\perp=\{f-\mathbb{E}_\mu f \ |\  f\in L^{2}(M,\Mz_n)\}$. Then, since $P(\ker(L)^\perp)\subset\ker(L)^\perp$ (namely, $P(\bf 1)=\bf 1$) we have that the following equvilance
%\masha{in general, I prefer not to use implication arrows in papers. Rather write it in several steps.}
\begin{align}\label{eq:operatorinequality}
&\ P^\dagger K_{\rho}P\leqslant C K_{P^\dagger \rho}
\\ \Longleftrightarrow &\  K_{P^\dagger\rho}^{-1/2}P^\dagger K_{\rho}PK_{P^\dagger\rho}^{-1/2}\leqslant C
\\ \Longleftrightarrow &\  \norm{K_{P^\dagger\rho}^{-1/2}P^\dagger K_{\rho}^{1/2}}{}\leqslant \sqrt{C}\nonumber
\\ \Longleftrightarrow &\  K_{\rho}^{1/2}P^\dagger K_{P^\dagger\rho}^{-1}PK_{\rho}^{1/2}\leqslant C\nonumber
\\ \Longleftrightarrow &\  P K_{P^\dagger \rho}^{-1}P^\dagger\leqslant CK_{\rho}^{-1}.
\end{align}
Recall that we can represent $M_\rho(f)$ as the double operator integral
\begin{align*}M_\rho(f)=\int_{0}^1\rho^{s}f\rho^{1-s}ds= \int_{0}^\infty\int_{0}^\infty\frac{\lambda-\mu}{\ln \lambda-\ln \mu} d{E_\lambda(\rho)}(f) d{E_\mu(\rho)}\pl,
\end{align*}
where $E_\lambda(\rho)$ is the spectral resolution of $\rho$. Then
\begin{align*}
M_\rho^{-1}(f)= \int_{0}^\infty\int_{0}^\infty\frac{\ln \lambda-\ln \mu}{\lambda-\mu} d{E_\lambda(\rho)}(f) d{E_\mu(\rho)}=\int_{0}^\infty (\rho+s{\bf 1})^{-1} f(\rho+s{\bf 1})^{-1}ds.
\end{align*}
Therefore by the noncommutative chain rule \eqref{eq:chain}
\[
K_{\rho}(\ln\rho)=\nabla^*M_\rho\nabla(\ln\rho)=\nabla^*M_\rho M_\rho^{-1}\nabla(\rho)=-\cL(\rho).
\]
Then for mean zero elements $-\cL\rho\in \ker(L)^\perp$ we have
\begin{align*}
 &\lan (-\cL\rho), K_{\rho}^{-1}(-\cL\rho)\ran_{\tr}
 \\ =& \lan (-\cL\rho), K_{\rho}^{-1}K_{\rho}(\ln\rho)\ran_{\tr}
 \\ =& \lan -\cL\rho, \ln\rho\ran_{\tr}
 \\= &I(\rho)\pl.
\end{align*}
Similarly, since $[P,-\cL]=0$
\begin{align*}
 &\lan (-\cL\rho), P K_{P^\dagger\rho}^{-1}P^\dagger(-\cL\rho)\ran
 \\ =& \lan -\cL P^\dagger\rho, K_{P^\dagger\rho}^{-1}(-\cL P^\dagger\rho)\ran
 \\ =& I(P^\dagger\rho).
\end{align*}
Thus by\eqref{eq:operatorinequality} we have
\[
I(P^\dagger\rho)\leqslant CI(\rho)\pl.
\]
That completes the proof.
\end{proof}

We now show that the Fisher information inequality in Lemma \ref{lemma:fisher} is generally sufficient to derive an CMLSI.

\begin{theorem}\label{thm:integralable}
Let $(M,g)$ be a Riemannian manifold with a sub-Laplacian $L=-\sum_{i=1}^kX_i^*X_i$ and let $P_t=e^{tL}$ be the symmetric Markov semigroup generated by $L$. Suppose there exists a positive function $C:[0,\infty)\to \mathbb{R}_+$ such that
\begin{enumerate}
\item[i)] for all scalar-valued functions $f\in \operatorname{dom}(\Gamma)$ and $t\geqslant  0$,
\begin{align} \Gamma(P_t f, P_t f)\leqslant C(t)P_t \Gamma(f, f)\label{eq:GE}\end{align}
\item[ii)] $\displaystyle \kappa:=\int_{0}^\infty C(s)ds< \infty$.
\end{enumerate}
Then $P_t$ satisfies $\operatorname{CMLSI}\left( \lambda \right)$ with $\lambda=\frac{1}{2\kappa}$.
\end{theorem}
\begin{proof}
Assumption i) allows us to apply Proposition \ref{prop:gradient} and Lemma \ref{lemma:fisher} to see that for any bounded and invertible matrix-valued density function $\rho$,
\[
I(P_t\rho)\leqslant C(t)I(\rho)\pl.
\]
Then, using the identity
\[
I(P_t\rho)=-\frac{d}{dt}\left( D(P_t\rho||\mathbb{E}_\mu \rho) \right)
\]
and $P_0(\rho)=\rho, \displaystyle \lim_{t\to\infty }P_t(\rho)=\mathbb{E}_\mu \rho$,

\begin{align*}
& \int_{\Omega} \tr(\rho\ln \rho-\mathbb{E}_\mu \rho\ln \mathbb{E}_\mu \rho) d\mu=D(\rho||\mathbb{E}_\mu \rho)
\\
& =\int_{0}^\infty I(P_t\rho)dt\leqslant \int_{0}^\infty C(t) I(\rho)dt=\kappa I(\rho).
\end{align*}
This proves $\operatorname{CMLSI}$ for bounded and invertible $\rho$. The general case follows by approximation, see e.g. \cite[Appendix]{BrannanGaoJunge2022}.
\end{proof}

Theorem~\ref{thm:integralable} can be regarded as a matrix-valued extension of Bakry-{\'E}mery's approach. Indeed, in \cite{BakryEmery1985} they proved that if $ \operatorname{Ricci}\geqslant \lambda$ then for a scalar-valued function
\[
\Gamma(P_t f, P_t f)\leqslant e^{-2\lambda t}P_t \Gamma(f, f),
\]
and therefore
\[
I(P_t f)\leqslant e^{-2\lambda t}I(f).
\]
Thus $P_{t}$ satisfies $\operatorname{CMLSI}\left( \lambda \right)$, where
\[
\lambda=\frac{2}{\kappa}=2\left( \int_{0}^\infty  e^{-2\lambda t} dt \right)^{-1}.
\]
This approach works if and only if $e^{-2\lambda t}$ is integrable, in other words, the lower curvature bound $\lambda$ is positive. Recently this method has been extended to a non-positive curvature regime, where
\[
I(P_t f)\leqslant e^{2\lambda t}I(f) \text{ for some } \lambda >0,
\]
and the upper bound is not integrable. The idea is to use $L^{\infty}$-mixing time defined for $0<\eps<1$
\[
t(\eps):=\inf \{t>0: \Vert P_t-\mathbb{E}_\mu \Vert_{L^{1}(\Omega)\to L^{\infty}(\Omega)}
\leqslant \eps\}.
\]
The $L^{\infty}$-mixing time approach was used in \cite{diaconis1996logarithmic} to prove an LSI via hypercontractivity. It was proved in \cite{BrannanGaoJunge2022} that at time $t(\eps)$,
 for any matrix-valued density $\rho$ the relative entropy \eqref{eq:re} from $P_{t(\eps)}\rho$ to the mean $\mathbb{E}_\mu \rho$ is at most $\eps$ times the initial relative entropy, i.e.
\begin{align}\label{eq:epsclose}
D(P_{t(\eps)}\rho||\mathbb{E}_\mu \rho)\leqslant \eps D(\rho||\mathbb{E}_\mu \rho)\pl.
 \end{align}
Using this, we have the following version of Theorem \eqref{thm:integralable}.
\begin{theorem}
Let $(M,g)$ be a Riemannian manifold with a sub-Laplacian $\cL=-\sum_{i=1}^kX_i^*X_i$ and let $P_t=e^{t\cL}$ be the symmetric Markov semigroup generated by $L$. Suppose
\begin{enumerate}
\item[i)] for some $0<\eps<1$, the $L^{\infty}$-mixing time is finite
\[
t(\eps):=\inf \{t>0:  \Vert P_t-\mathbb{E}_\mu\Vert_{L^1(\Omega)\to L^{\infty}(\Omega)} \leqslant \eps\}<\infty.
\]
\item[ii)] there exists a positive function $C:[0,t(\eps))\to \mathbb{R}_+$ such that
for all scalar-valued functions $f\in \operatorname{dom}(\Gamma)$ and $0\leqslant t\leqslant t(\eps)$,
\begin{align} \Gamma(P_t f, P_t f)\leqslant C(t)P_t \Gamma(f, f)\label{eq:GE2}\end{align}
and $\displaystyle \kappa_{\eps}:=\int_{0}^{t(\eps)} C(s)ds$ is finite.
\end{enumerate}
Then $P_t$ satisfies $\operatorname{CMLSI}\left( \lambda_{\eps} \right)$ with $\displaystyle \lambda_{\eps}=\frac{(1-\eps)}{2\kappa_{\eps}}$.
\end{theorem}
\begin{proof} The proof is similar to the proof of Theorem \ref{thm:integralable}. Note that by \eqref{eq:epsclose}, we have
\[
D(\rho||\mathbb{E}_\mu \rho)-D(P_{t(\eps)}\rho||\mathbb{E}_\mu \rho)\geqslant  (1-\eps) D(\rho||\mathbb{E}_\mu \rho).
\]
Then
\begin{align*}
(1-\eps)D(\rho||\mathbb{E}_{\mu} \rho)\ \leqslant\  & D(\rho||\mathbb{E}_\mu \rho)-D(P_{t(\eps)}\rho||\mathbb{E}_\mu \rho)
\\ =&\int_{0}^{t(\eps)} I(P_t\rho)dt
\\ \leqslant &\int_{0}^{t(\eps)} C(t) I(\rho)dt=\kappa I(\rho),
\end{align*}
which proves the claim.
\end{proof}

The $L^{\infty}$-mixing time $t(\eps)$ is finite whenever the semigroup satisfies a  Poincar\'e inequality $(\operatorname{PI})$ and ultracontractivity.
\begin{align*}
& \Vert P_t-\mathbb{E}_{\mu}\Vert_{L^{2}(\Omega, \mu)\to L^{2}(\Omega, \mu)} \leqslant e^{-\gamma t} \tag{$\operatorname{PI}\left( \gamma \right)$}
\\
&\Vert P_{t_0}\Vert_{L^1(\Omega,\mu)\to L^{\infty}(\Omega,\mu)} \leqslant C \text{ for some } t_0 \geqslant  0.  \tag{Ultracontractivity}
\end{align*}
Indeed, for $t\geqslant  t_0$
\begin{align*}
&\Vert P_t-\mathbb{E}_{\mu}\Vert_{L^{1}(\Omega,\mu)\to L^{\infty}(\Omega,\mu)}
\\
& \leqslant \Vert P_{\frac{t_0}{2}}\Vert_{L^{1} \to L^{2}} \cdot \Vert P_{t-t_0}-\mathbb{E}_{\mu}\Vert_{L^{2}\to L^{2}} \cdot
\Vert P_{\frac{t_0}{2}}\Vert_{L^{2}\to L^{\infty}}
\\
& \leqslant \Vert P_{t_0}\Vert_{L^{1}\to L^{\infty}} \cdot \Vert P_{t-t_0}-\mathbb{E}_{\mu}\Vert_{ L^{2}\to L^{2}}\leqslant Ce^{-\gamma(t-t_0)}.
\end{align*}
Therefore, $t(\eps)\leqslant \frac{\ln C-\ln \eps}{\gamma} +t_0$.

We now show that under H\"ormander's condition, the gradient estimate \eqref{eq:gradient2} for any short period of time $t$ is sufficient for $\operatorname{CMLSI}$. We recall a lemma from \cite[Lemma 2.1 \& Lemma 2.2]{GaoRouze2021}.
\begin{lemma}\label{lemma:key}
Let $\rho,\sigma\in L^{\infty}(M,\Mz_n)$ be two matrix-valued density functions. Denote the multiplication operator $M_\rho^{-1}: L^{\infty}(M,\Mz_n)\to L^{\infty}(M,\Mz_n)$ and the associate weighted $L^{2}$-norm as follows
\begin{align*}
 &M_\rho^{-1}(f)=\int_{0}^\infty(\rho+s{\bf 1})^{-1}f(\rho+s{\bf 1})^{-1}ds\\
 &\norm{f}{\rho^{-1}}^2=\lan f, f \ran_{\rho^{-1}}:=\lan f, M_\rho^{-1} f\ran_{\tr}=\int_{0}^\infty\tr\Big( f^*(\rho+s{\bf 1})^{-1} f (\rho+s{\bf 1})^{-1} \Big)ds\pl.
\end{align*}
Then,
\begin{enumerate}
\item[i)] if $\rho\leqslant C\sigma$, $C\norm{f}{\rho^{-1}}\geqslant   \norm{f}{\sigma^{-1}}$
\item[ii)] $D(\rho||\mathbb{E}_\mu \rho)\leqslant \norm{\rho-\mathbb{E}_\mu \rho}{\mathbb{E}_\mu \rho^{-1}}$
\end{enumerate}
\end{lemma}
\begin{proof} The proof is almost identical to the finite-dimensional matrix case in \cite{GaoRouze2021}.
\end{proof}

\begin{theorem}\label{thm:main}
Let $(M,g)$ be a Riemannian manifold with $\mu$ being a smooth probability measure. Let
$\cL=-\sum_{i=1}^kX_i^*X_i$ be a sub-Laplacian and denote by $P_t=e^{t\cL}$ the symmetric Markov semigroup generated by $\cL$. Suppose
\begin{enumerate}
\item[i)] $\{X_i\}_{i=1}^k$ satisfies H\"ormander's condition.
\item[ii)] there exists some $0<t_0<1$ and a positive function $C:[0,t_0)\to \mathbb{R}_+$ such that for all scalar-valued functions $f\in \operatorname{dom}(\Gamma)$ and $0\leqslant t\leqslant t_0$,
\begin{align} \Gamma(P_t f, P_t f)\leqslant C(t)P_t \Gamma(f, f)\label{eq:GE3}
\end{align}
and $\displaystyle \kappa:=\int_{0}^{t_0} C(s)ds$ is finite.
\end{enumerate}
Then $P_t$ satisfies $\operatorname{CMLSI}$.
\end{theorem}
\begin{proof}
By Varopoulos' dimension condition, we have that for any positive scalar function $f$
\[
P_tf\leqslant \norm{P_tf}{\infty} {\bf 1}\leqslant  Ct^{-m}\norm{f}{1} {\bf 1}=   Ct^{-m} \mathbb{E}_\mu f\pl,\pl 0< t\leqslant 1
\]
for some $C$ and $m$.
Then $Ct^{-m} \mathbb{E}_{\mu}-P_t$ is a positive map from $L^{\infty}(M)$ to $L^{\infty}(M)$. Note that any positive map onto $L^{\infty}(M)$ is automatically completely positive, because positivity in $L^{\infty}(M,\Mz_n)$ is a.s pointwise positive. Then, for any matrix-valued density $\rho$
\[
P_t\rho\leqslant Ct^{-m} (\mathbb{E}_\mu\rho){\bf 1}\pl,\pl 0< t\leqslant 1
\]
as matrix-valued functions in $L^{\infty}(M,\Mz_n)$, where the right-hand side stands for the constant function of the value $\mathbb{E}_\mu\rho=\int_M \rho \ d\mu$.
For a matrix-valued density $\rho$ such that $D(\rho||\E_\mu\rho)<\infty$, we denote $\rho_t:=P_t\rho$. Then for any $0<t\leqslant 1 $
\begin{align}
D(\rho_t||\mathbb{E}_\mu\rho)\leqslant  & \norm{\rho_t-\mathbb{E}_\mu\rho}{\mathbb{E}_\mu\rho^{-1}}\label{eq:step1}
\\ \leqslant & \frac{1}{\gamma}\norm{\nabla\rho_t}{\mathbb{E}_\mu\rho^{-1}} \label{eq:step2}
\\ \leqslant & \frac{t^m}{C\gamma}\norm{\nabla\rho_t}{\rho_t^{-1}}= \frac{t^m}{C\gamma }I(\rho_t) \label{eq:step3}.
\end{align}
Here \eqref{eq:step1} uses part  ii) of Lemma \ref{lemma:key}, \eqref{eq:step2} uses the spectral gap $\gamma$ and \eqref{eq:step3} uses part i) of Lemma \eqref{lemma:key}  and $\rho_t=P_t\rho\leqslant Ct^{-m} E_\mu\rho$.
Denote $h(t)=D(\rho_t||\mathbb{E}_\mu\rho)$. Then the above inequality shows that
\[
h(t)\leqslant -\frac{t^m}{C\gamma }h^{\prime}(t)\pl, \pl  0< t\leqslant 1
\]
Consider the differential equation
\[
s^{\prime}(t)=-\frac{t^m }{C\gamma} s(t),
\]
whose solution is
\[
s(t)=e^{-\frac{ t^{m+1}}{\gamma C(m+1)}}s(0).
\]
Thus by Gr\"{o}nwall's inequality
\[
D(\rho_{t}||\mathbb{E}_\mu(\rho))=h(t)\leqslant e^{-\frac{ t^{m+1}}{\gamma C(m+1)}}h(0)=D(\rho||\mathbb{E}_\mu(\rho)) \pl, \pl  0< t\leqslant 1\pl.
\]
This implies that for any matrix-valued density $\rho$
\[
D(\rho||\mathbb{E}_\mu(\rho))-D(\rho_{t_0}||\mathbb{E}_\mu(\rho))\geqslant  (1-e^{-\frac{ t^{m+1}}{\gamma C(m+1)}})D(\rho||\mathbb{E}_\mu(\rho)).
\]
Then using assumption \eqref{eq:GE3},
\begin{align*}
(1-e^{-\frac{ t^{m+1}}{\gamma C(m+1)}})D(\rho||\mathbb{E}_\mu(\rho))
\leqslant
& (D(\rho||\mathbb{E}_\mu(\rho))-D(\rho_{t_0}||\mathbb{E}_\mu(\rho))
\\
\leqslant & \int_{0}^{t_0} I(\rho_t)dt
\\
\leqslant &
\int_{0}^{t_0} C(t) I(\rho)dt
\\= & \kappa I(\rho)
\end{align*}
which proves $\operatorname{CMLSI}\left( \lambda \right)$ with
\begin{align*}
\la=& \kappa^{-1}(1-e^{-\frac{ t^{m+1}}{\gamma C(m+1)}})\pl. \qedhere
\end{align*}
\end{proof}

\subsection{Applications to the sub-Laplacian on $\operatorname{SU}\left( 2 \right)$}
\label{sec:su2}
We now apply our results to the canonical  sub-Laplacian on the special unitary group $\operatorname{SU}\left( 2 \right)$. Recall the skew-Hermitian matrices
\[
X=\left[\begin{array}{cc} 0& 1\\ -1& 0
\end{array}\right],Y=\left[\begin{array}{cc} 0& i\\ i& 0
\end{array}\right], Z=\left[\begin{array}{cc} i& 0\\ 0& -i
\end{array}\right]
\]
span its Lie algebra over $\mathbb{R}$. The group $G=\operatorname{SU}\left( 2 \right)$ is isomorphic to $3$-sphere $\mathbb{S}^3$ via the following parametrization
\[
G=\{ cI+xX+yY+zZ:  c^2+x^2+y^2+z^2=1, x,y,x, c \in \mathbb{R} \}.
\]
The Lie algebra is $\mathfrak{g}=\operatorname{span}_{\mathbb{R}}\{X,Y,Z\}$ with Lie bracket rules as
\begin{align}[X,Y]=2Z\ , [Y,Z]=2X\ , [Z,X]=2Y\pl. \label{eq:lie}\end{align}
Suppose that $\mathfrak{g}$ is equipped with the left-invariant metric such that $\{ X,Y,Z \}\subset \mathfrak{g}$ forms an orthonormal basis, the corresponding Laplacian (Casimir operator) is
\[
\Delta=X^2+Y^2+Z^2.
\]
It is known that $\operatorname{SU}\left( 2 \right)$ has the constant Ricci curvature $2$. Then by the complete Bakry-{\'E}mery theorem \cite[Theorem 3.4]{JLLR} the heat semigroup $T_t=e^{-\Delta t}$ satisfies $\operatorname{CMLSI}\left( 2 \right)$.

We are interested in the sub-Riemannian setting. The canonical sub-Riemannian structure is given by $H=\{X,Y\}$ as a generating set of $\mathfrak{g}$ since $[X,Y]=2Z$ satisfies H\"ormander's condition. The horizontal sub-Laplacian is
\[
\cL=X^2+Y^2.
\]
The semigroup $P_t=e^{\cL t}$ on $\operatorname{SU}\left( 2 \right)$ has been intensively studied. In particular, Baudoin and Bonnefont in \cite{BaudoinBonnefont2009} proved that for all $p>1$, the following $L^{p}$-gradient estimate  holds for all scalar-valued function $f$,
\begin{align}\label{eq:baudoin}
\Gamma(P_t f,P_t f)^{\frac{p}{2}}\leqslant C_p e^{-2p t}P_t(\Gamma(f, f)^{\frac{p}{2}}),
\end{align}
where $C_p$ is a constant depending on $p$. Applying our Theorem \ref{thm:integralable} for $p=2$, we obtain

\begin{theorem}
The semigroup $P_t=e^{\cL t}$ generated by the sub-Laplacian $\cL=X^2+Y^2$ satisfies $\operatorname{CMLSI}\left( \frac{C_2}{8} \right)$, where $C_2$ is the constant in \eqref{eq:baudoin} for $p=2$.
\end{theorem}

We note that this is the first example that CMLSI is obtained for sub-Laplacians.
The CMLSI constant for sub-Laplacians has direct applications to quantum Markov semigroup via its representations. The representation theory of $\operatorname{SU}(2)$ gives the well-known spin structure of quantum mechanics, where any irreducible representation of $\operatorname{SU}(2)$ is indexed by an integer $m\in \mathbb{N}^+$.
Let $\phi_m:\mathfrak{su}(2)\to i(\mathbb{M}_{m})_{\operatorname{sa}}$ be the Lie algebra homomorphism induced by the $m$-th irreducible representation, and let $\{\ket{j}| j=1,\cdots,m \}$
be the orthonormal basis consisting of eigenfunctions $\eta_m(Z)$.  Denote $X_m:=\phi_m(X)$, and similarly for $Y_m$ and $Z_m$. Under the normalization of \eqref{eq:lie},
\begin{align*}
& X_m\ket{j}=\sqrt{(j-1)(m-j+1)}\,\ket{j-1}-\sqrt{(j+1)(m-j-1)}\,\ket{j+1},
\\
& Y_m\ket{j}=i\sqrt{(j-1)(m-j+1)}\,\ket{j-1}+i\sqrt{(j+1)(m-j-1)}\,\ket{j+1},
\\
& Z_m\ket{j}=(m-2j+1)\,\ket{j}.
\end{align*}
We consider the quantum Markov semigroup
\begin{align*}
& S_t=e^{\cL_m t}:\Mz_m\to \Mz_m\pl,
\\
& \cL_m(\rho)=[X_m,[X_m, \rho]]+[Y_m,[Y_m, \rho]].
\end{align*}
This semigroup can be viewed as a representation of classical Markov semigroup $P_t:L^{\infty}(G)\to L^{\infty}(G)$. Indeed, let $\pi_m: G\to \Mz_{m}$ be the $m$th irreducible representation of $G=\operatorname{SU}\left( 2 \right)$. The following transference diagram holds

\begin{equation}\label{dia:transference}
\begin{tikzcd}
L^{\infty}(G, \Mz_m)  \arrow{r}{P_t\ten \id_{\Mz_m}}  & L^{\infty}(G, \Mz_m)
 \\
\Mz_m  \arrow{u}{\alpha} \arrow{r}{S_{t}} & \Mz_m  \arrow{u}{\alpha},
\end{tikzcd}
\end{equation}

where the map $\al$ is given by
\begin{align*}
& \al: \Mz_m\to L^{\infty}(G, \Mz_m)\pl,
\\
& \al(\rho)(g)=\pi_m(g)(\rho)\pi_m(g)^*\pl.
\end{align*}
Note that $\al$ is an algebra homomorphism, which is an embedding of the matrix algebra $\Mz_m$ into the matrix-valued functions $L^{\infty}(G, \Mz_m)$ on $G$. Under this embedding, the quantum Markov semigroup $S_t$ is exactly the restriction of the matrix-valued extension $P_t\ten \id_{\Mz_m}$ on the image of $\al(\Mz_m)$. Such a transference relation holds for any (projective) unitary representation, which yields the following dimension-free CMLSI estimate.

\begin{theorem}
Let $S_t=e^{\cL_m t}:\Mz_m\to \Mz_m\pl, \cL_m(\rho)=[X_m,[X_m, \rho]]+[Y_m,[Y_m, \rho]]$ be the transference semigroups above. Then for all $m\geqslant  1$, $S_t$ satisfies $\operatorname{CMLSI}\left( \frac{C_2}{8}\right)$.
\end{theorem}

\begin{proof} This follows from the diagram \eqref{dia:transference} as in \cite[Section 4]{gao2020fisher}.
\end{proof}

It is clear from the diagram \eqref{dia:transference} that CMLSI constant for matrix-valued functions in $L^{\infty}(G,\Mz_m)$ is crucial for the transference to quantum Markov semigroup on matrix algebras, while the LSI or MLSI only for the scalar-valued functions is not enough.

\section{Ornstein–Uhlenbeck semigroup on Lie groups}\label{sec:ou}
In this section, we discuss the matrix-valued modified Log-Sobolev constants for the heat kernel measure on Lie groups. Before the matrix-valued case, we first review a standard approach for scalar-valued functions.
Let $(M,g, d_g)$ be a Riemannian manifold with volume form $d_g$ and let $\nabla$ and $\Delta=\nabla\cdot \nabla$ be the gradient operator and Laplace-Beltrami operator respectively. Denote $H_t=e^{\Delta t}$ as the heat semigroup. Recall that Bakry and {\'E}mery \cite{BakryEmery1985} proved that $\operatorname{Ricci}(M)\geqslant \lambda$ for some $\lambda\in \mathbb{R}$ if and only if the following $p=1$ gradient estimate
\begin{align}\label{eq:p=1}
|\nabla H_t f|\leqslant C(t) H_t(|\nabla f|)\pl, \text{ for all } f\in C_c^\infty(M), t>0
\end{align}
holds for $C(t)=e^{-\lambda t}$. Let $p_t^x$ be the heat kernel measure at some point $x\in M$ such that
\[
H_tf(x)=\int_{M}fdp_t^x\pl.
\]
It is well known that the gradient estimate \eqref{eq:p=1} for $p=1$ implies the LSI for $p^x_t$

\begin{align}\label{eq:LSIou}
\int f^2\log f^2 dp^x_t-\left( \int f^2 dp^x_t \right) \log \left( \int f^2 dp^x_t \right)\leqslant 2t \int |\nabla f|^2 dp^x_t\pl
\end{align}
with the constant $\kappa=\int_{0}^t C(s)^2ds$. This further implies an MLSI, namely,  that for $g\geqslant 0$
\begin{align} \label{eq:MLSIou}
\int g\log g dp^x_t-\left( \int g dp^x_t \right) \log \left( \int g dp^x_t \right)\leqslant t \int \frac{|\nabla g|^2}{g} dp^x_t.
\end{align}
In terms of the semigroup, \eqref{eq:LSIou} and \eqref{eq:MLSIou} are equivalent to respectively the hypercontractivity and exponential entropy decay for the Ornstein–Uhlenbeck semigroup
\[
P_s=e^{L_{x,t} s}, L_{x,t}=\nabla^*\nabla,
\]
where $\nabla^*$ is the adjoint of $\nabla$ in $L^{2}(dp_t^x)$ and $L_{x,t}$ is symmetric with respect to the measure $p^x_t$.

The proof of \eqref{eq:LSIou} and \eqref{eq:MLSIou} from \eqref{eq:p=1} are now standard (c.f. \cite[Theorem 6.1]{BakryBaudoinBonnefontChafai2008}). Indeed, \eqref{eq:MLSIou} follows from the function inequality at point $x\in M$,
\begin{align}\label{eq:function}
H_t(g\log g)-H_t(g)\log H_t(g)\leqslant \frac{\kappa}{2}H_t\left( \frac{\nabla g}{g}\right).
\end{align}
Fix $t>0$, and for $g\in C^\infty_c(M)$ we define
\[
G:[0,t]\to C^\infty(M)\pl,\pl G(s)=H_s(H_{t-s}g\log H_{t-s}g).
\]
Note that
\begin{align*}
H_t(g\log g)=G(t)\pl,   H_{t-s}g\log H_{t-s}g=G(0)\pl.
\end{align*}
Then
\begin{align*}
H_t(g\log g)-H_{t}g\log H_{t}g=G(t)-G(0)
=\int_{0}^t \partial_s G(s)ds \pl.
\end{align*}
The derivative is given by
\begin{align*}
\partial_s G(s)=&\Delta H_s\left( H_{t-s}g\log  H_{t-s}g \right)-H_s\left( (\Delta H_{t-s})g \log  H_{t-s}g \right)-H_s( \Delta H_{t-s}g )
\\
=& H_s\left( \Delta (H_{t-s}g\log  H_{t-s}g) -(\Delta H_{t-s}g )\log  H_{t-s}g -\Delta H_{t-s}g \right)
\\
=& H_s\left( \Delta (H_{t-s}g\log  H_{t-s}g) -(\Delta H_{t-s}g )\log  H_{t-s}g - \Delta H_{t-s}g \right)
\\
= &2H_s\left( \nabla H_{t-s} g \cdot \nabla(\log  H_{t-s}g) \right)
\\
=&2H_s\left( \frac{|\nabla H_{t-s} g|^2}{H_{t-s} g} \right)\pl,
\end{align*}
where we used the product rule for $\Delta=\nabla\cdot \nabla$
\[
\Delta (f\log  f)=\Delta (f)\log  f + f\Delta(\log  f) + 2\nabla f \nabla (\log f)\pl.
\]
Here comes the step using the gradient estimate for $p=1$ which gives that for any $s>0$,

\begin{align} \label{eq:fisherp1}
\frac{|\nabla H_{s} g|^2}{H_{s} g}\leqslant C(s)^2\frac{H_{s}(|\nabla  g|)^2}{H_{s} g}\leqslant  C(s)^2H_{s}(\frac{|\nabla  g|^2}{ g}),
\end{align}
where the second inequality uses the joint convexity of the bivariate function $(x,y)\mapsto \frac{x^2}{y}$. Then \eqref{eq:function} follows immediately from the semigroup property

\begin{align*}
H_t(g\log g)-H_{t}g\log H_{t}g\leqslant &2\int_{0}^tH_s\left( \frac{|\nabla H_{t-s} g|^2}{H_{t-s} g}\right)ds \leqslant 2C(s)^2\int_{0}^t H_{t-s}H_{s}\left(\frac{|\nabla  g|^2}{ g}\right)ds\\= &2(\int_{0}^t C(s)^2ds) H_{t}(\frac{|\nabla  g|^2}{ g}).
\end{align*}
To summarise, the key consequence the gradient estimate \eqref{eq:p=1} for $p=1$ used above is that
\begin{align}\label{eq:keyp=1}
|\nabla H_t f|\leqslant C(t) H_t(|\nabla f|)\pl \Rightarrow \pl \frac{|\nabla H_{t} g|^2}{H_{t} g}\leqslant  C(t)^2H_{t}(\frac{|\nabla  g|^2}{ g}) \pl.
\end{align}
This approach can be used to recover Gross' LSI for the Ornstein–Uhlenbeck semigroup on $\mathbb{R}^n$, as well as the Heisenberg group in \cite{BakryBaudoinBonnefontChafai2008} in the sub-Riemannian setting, for which the gradient estimate \eqref{eq:p=1}  for $p=1$ was proved by H.-Q.~Li in \cite{li2006estimation} with $C(t)\cong C$ being constant.

Given our Theorem \ref{thm:main} using the gradient estimate for $p=2$, it is natural to consider whether the gradient estimate for  $p=1$  can be used to tackle the matrix-valued case.
However, there are two obstructions. First, it is not clear whether $p=1$ gradient estimate extends to matrix-valued cases as in Proposition \ref{prop:gradient}. Note that for a matrix $A$, the absolute value $|A|$ should be interpreted as $\sqrt{A^*A}$ using functional calculus for positive operators. For a matrix-valued function $f\in C_c^\infty(M)$, we have the gradient vector $\nabla f=(X_1f, \cdots, X_nf)$ and $|\nabla f|= \sqrt{\sum_{i} (X_i f)^*(X_i f)}$. Secondly, we also lack of \eqref{eq:keyp=1} in the matrix-valued case as an analog of Lemma \ref{lemma:fisher} for $p=2$. These observations make it unclear whether \eqref{eq:p=1} implies a CMLSI for theOrnstein–Uhlenbeck semigroup.

In the following, we shall show that the above approach partially works for matrix-valued case if we have the exact commutation relation
\begin{align}\label{eq:interwine}
\nabla H_t f= e^{-\lambda t}H_t\nabla f.
\end{align}
Equation \eqref{eq:interwine} is also called an intertwining relation. The natural example of spaces for which \eqref{eq:interwine} holds are groups. Let $G$ be a Lie group equipped with a left-invariant metric. Let $X_1,\cdots,X_n$ be a O.N.B of the Lie algebra of left invariant vector field. It is well known that the Casimir element
$\Delta=\sum_{i}X_i^2$, which is also the Laplace-Beltrami operator, is a central element for the Lie algebra. Thus for the gradient operator $\nabla=(X_1,\cdots,X_n)$
\begin{align}\label{eq:interwine2}
&\nabla\Delta f=\Delta \nabla f\pl,
    \pl \nabla H_t f= H_t\nabla f\pl, \pl \forall f\in C_c^\infty(M).
\end{align}
 which is an intertwining relation \eqref{eq:interwine} with $\lambda=0$.
\begin{theorem}\label{thm:OU}
Let $G$ be a locally compact Lie group equipped with a left-invariant metric. Let $H_t=e^{\Delta t}$ be the heat semigroup and denote $p_t$ as the heat kernel measure (for the identity element).
Then for any positive matrix-valued function $g\in C_c^\infty(G, M_n)$
\[
\tr H_t(g\log g )-\tr (H_tg\log H_tg)  \leqslant 2t \tr H_{t} \Big(\int_{0}^\infty (\nabla g)(g+r{\bf 1})^{-1}(\nabla g)(g+r{\bf 1})^{-1}dr\Big)\pl,
\]
where ${\bf 1}$ is the constant matrix-valued function of identity operator.
In particular, the Ornstein–Uhlenbeck semigroup $P_s=e^{L_{t}s}$ satisfies $\frac{1}{4t}$-\text{CMLSI}, where $L_t=\nabla^*\nabla$ is the generator symmetric on $L^{2}(dp_t)$.
\end{theorem}

\begin{proof}
Let $g\in C_c^\infty(G, M_n)$ be a matrix-valued function on $G$ and $g(x)\geqslant 0$ for all $x\in G$. For the ease of notation, we write $H_tg:=(H_t\ten \id_{M_n})g$ and $\Delta g= (\Delta \ten \id_{M_n})g$.
Fix $t>0$ and define the function
\[
G:[0,t]\to C^\infty(M)\pl,\pl G(s)=\tr\Big(H_{s}(H_{t-s}g\log H_{t-s}g )\Big).
\]
Then
\begin{align*} \tr H_t(g\log g)-\tr (H_{t}g\log H_{t}g)=G(t)-G(0)
=\int_{0}^t \partial_s G(s)ds \pl.\end{align*}
Recall that for a smooth scalar function $\phi:(a,b)\to \mathbb{R}$ on a interval $(a,b)$ and matrix-valued function $s\mapsto g(s)$ with $\text{spec}g(s)\subset (a,b)$,
\begin{align}\frac{d}{ds}\tr(\phi(g(s)))=\tr(\phi^{\prime}(g(s))g^{\prime}(s))\pl. \label{eq:de}\end{align}
The derivative with respect to $s$ is given by
\begin{align*}\partial_s G(s)=&\Delta H_s\Big(\tr  \left( H_{t-s}g\log  H_{t-s}g \right)\Big)+H_s\Big(\tr(-\Delta H_{t-s}g (\log  H_{t-s}g+1))\Big)\\
=&H_s \Big(\tr  \big( \Delta (H_{t-s}g\log  H_{t-s}g)-\Delta (H_{t-s}g) \log  H_{t-s}g-\Delta (H_{t-s}g) \big)\Big)\pl,
\end{align*}
where we used \eqref{eq:de} with $\phi(x)=x\log x$ and the fact that $H_t \tr =\tr H_t$ commute (more precisely, $H_t \tr =\tr (H_t\ten \id_{M_n})$).
By the product rule for $\Delta=\nabla\cdot \nabla$,
\begin{align*}
\Delta (f\log  f)=&\Delta (f)\log  f + f\Delta(\log  f) + 2\nabla f \nabla (\log f).
\end{align*}
Using the noncommutative derivative,
\begin{align*}&\nabla(\log f)=\int_{0}^\infty(f+r{\bf 1})^{-1} (\nabla f)(f+r{\bf 1})^{-1}dr, \\
&\nabla\big((f+r{\bf 1})^{-1}\big)=-(f+r{\bf 1})^{-1} (\nabla f)(f+r{\bf 1})^{-1}.
\end{align*}
we have
\begin{align*}\Delta(\log  f)=&\nabla\cdot \Big(\int_{0}^\infty(f+r{\bf 1})^{-1} (\nabla f)(f+r{\bf 1})^{-1}\Big)\\
=&\int_{0}^\infty(f+r{\bf 1})^{-1} (\Delta f)(f+r{\bf 1})^{-1}dr\\ &-2 \sum_{i}\int_{0}^\infty(f+r{\bf 1})^{-1} (X_i f)(f+r{\bf 1})^{-1}(X_i f)(f+r{\bf 1})^{-1}dr
\end{align*}
Then
\begin{align}\label{eq:4terms}
&\tr\Big(\Delta (f\log  f)-(\Delta f) \log  f-\Delta f\Big)\notag
\\
=&\tr\Big(f\Delta(\log  f) + 2\nabla f \nabla (\log f)-\Delta f\Big)\notag
\\
=&\tr\Big(\int_{0}^\infty f(f+r{\bf 1})^{-1} (\Delta f)(f+r{\bf 1})^{-1}dr\Big)\notag
\\ &-2 \sum_{i}\tr \Big(\int_{0}^\infty f(f+r{\bf 1})^{-1} (X_i f)(f+r{\bf 1})^{-1}(X_i f)(f+r{\bf 1})^{-1}dr\Big) \notag
\\ &+2 \sum_{i}\tr \Big(\int_{0}^\infty (X_i f)(f+r{\bf 1})^{-1}(X_i f)(f+r{\bf 1})^{-1}dr\Big)
\notag
\\ &-\tr\Big(\Delta f\Big).
\end{align}
The first and the last term cancel out
\[
\tr\Big(\int_{0}^\infty f(f+r{\bf 1})^{-1} (\Delta f)(f+r{\bf 1})^{-1}dr\Big)=\tr\Big(\big(\int_{0}^\infty f(f+r{\bf 1})^{-1}(f+r{\bf 1})^{-1}dr\big) \Delta f\Big)=\tr\Big( \Delta f\Big),
\]
where we used the tracial property $\tr(AB)=\tr(BA)$ and the integral identity
\[
\int_{0}^\infty (A+r{\bf 1})^{-1}A(A+r{\bf 1})dr={\bf 1}
\]
for a positive operator $A$. Also, the second term is always negative. Note that for positive $f$, $(X_i f)^*=X_i f$ is self-adjoint. Then for any $s>0$,
\begin{align*}
&\tr \Big( f(f+r{\bf 1})^{-1} (X_i f)(f+r{\bf 1})^{-1}(X_i f)(f+r{\bf 1})^{-1}\Big)
\\
=&\tr \Big( f^{1/2}(f+r{\bf 1})^{-1} (X_i f)(f+r{\bf 1})^{-1}(X_i f)(f+r{\bf 1})^{-1}f^{1/2}\Big)
\\ =&
\tr \Big(\big( f^{1/2}(f+r{\bf 1})^{-1} (X_i f)(f+r{\bf 1})^{-\frac{1}{2}} \big) \big((f+r{\bf 1})^{-\frac{1}{2}}(X_i f)(f+r{\bf 1})^{-1}f^{1/2}\big)\Big)\geqslant 0.
\end{align*}
Applying the above estimate to $f=H_{t-s}g$, we see that
\begin{align*}
&\tr  \big( \Delta (H_{t-s}g\log  H_{t-s}g)-\Delta (H_{t-s}g) \log  H_{t-s}g-\Delta (H_{t-s}g) \big)
\\ \leqslant & 2 \sum_{i}\tr \Big(\int_{0}^\infty (X_i H_{t-s}g)(H_{t-s}g+r{\bf 1})^{-1}(X_i H_{t-s}g)(H_{t-s}g+r{\bf 1})^{-1}dr\Big)
\\ =& 2 \tr \Big(\int_{0}^\infty (\nabla H_{t-s}g)(H_{t-s}g+r{\bf 1})^{-1}(\nabla H_{t-s}g)(H_{t-s}g+r{\bf 1})^{-1}dr\Big)
\\ \overset{(1)}{=}& 2 \tr \Big(\int_{0}^\infty (H_{t-s}\nabla g)(H_{t-s}g+r{\bf 1})^{-1}(H_{t-s}\nabla g)(H_{t-s}g+r{\bf 1})^{-1}dr\Big)
\\ \overset{(2)}{\leqslant} & 2 H_{t-s} \tr \Big(\int_{0}^\infty (\nabla g)(g+r{\bf 1})^{-1}(\nabla g)(g+r{\bf 1})^{-1}dr\Big)\pl.
\end{align*}
Here, (1) uses the commutation relation $\nabla H_{t}g= H_{t}\nabla g$, and (2) uses the joint convexity of the trace function (see c.f. \cite[Eq. (7)]{petz1996monotone})
\[
(A,B)\in (M_n)_+\times (M_n)_+\mapsto \tr\Big(\int_{0}^\infty A(B+r{\bf 1})^{-1}A(B+r{\bf 1})^{-1}dr\Big).
\]
Finally, using the semigroup property we see that
\begin{align} \label{eq:assertion}
&\tr H_t(g\log g)-\tr (H_{t}g\log H_{t}g)
=\int_{0}^t \partial_s G(s)ds \\
\leqslant
& 2\int_{0}^t H_s H_{t-s} \tr \Big(\int_{0}^\infty (\nabla g)(g+r{\bf 1})^{-1}(\nabla g)(g+r{\bf 1})^{-1}dr\Big)ds \notag
\\
=& 2\int_{0}^t H_{t} \tr \Big(\int_{0}^\infty (\nabla g)(g+r{\bf 1})^{-1}(\nabla g)(g+r{\bf 1})^{-1}dr\Big)ds
\notag
\\= &2t \tr H_{t} \Big(\int_{0}^\infty (\nabla g)(g+r{\bf 1})^{-1}(\nabla g)(g+r{\bf 1})^{-1}dr\Big)
\pl.\end{align}
To see that this implies a CMLSI, we note that $H_tf(e)=\mathbb{E}_{p_t} f$  at the identity element $e$ and the Fisher information for a positive matrix-valued function $g\in C_c^\infty(G, M_n)$ is
\[
I(g)=\int_{G}\tr\Big(\int_{0}^\infty (\nabla g)(g+r{\bf 1})^{-1}(\nabla g)(g+r{\bf 1})^{-1}dr\Big)dp_t.
\]
Then \eqref{eq:assertion} at the identity element $x=e$ gives
\begin{align*} D(g||\mathbb{E}_{p_t} g)=&
\int\tr(g\log g)-\tr(\mathbb{E}_{p_t} g \log\mathbb{E}_{p_t} g)dp_t
\\=&\tr H_t(g\log g)-\tr (H_{t}g\log H_{t}g)|_{x=e}
\\ \leqslant &2t \tr H_{t} \Big(\int_{0}^\infty (\nabla g)(g+r{\bf 1})^{-1}(\nabla g)(g+r{\bf 1})^{-1}dr\Big)|_{x=e}
\\ = &2t \int_G \tr \Big(\int_{0}^\infty (\nabla g)(g+r{\bf 1})^{-1}(\nabla g)(g+r{\bf 1})^{-1}dr\Big)dp_t
\\ = &2t I(g)
\pl.\end{align*}
That completes the proof.
\end{proof}

\begin{rem}Comparing to the scalar-valued case \eqref{eq:MLSIou}, our constant differs by a factor of $2$. This is because for a scalar-valued function $f$, we can find \eqref{eq:4terms} explicitly
\begin{align*}&-2\Big(\int_{0}^\infty f(f+r{\bf 1})^{-3} |\nabla f|^2 dr\Big)+2\Big(\int_{0}^\infty  (f+r{\bf 1})^{-2} |\nabla f|^2 dr\Big)\\
=&- f^{-1}|\nabla f|^2 +2f^{-1} |\nabla f|^2 = f^{-1} |\nabla f|^2\pl,
\end{align*}
which gains a factor of $2$ compared to the matrix-valued case.
\end{rem}

\section{Final discussion}\label{s.Final}
We end our discussion with some remarks on connection to the noncommutative setting and extendability of our results.

The assertion (i) in the key Proposition \ref{prop:gradient} is closely related to the \emph{$\Gamma_2$-condition} introduced by M.~Junge and Q.~Zeng \cite{JZ} for noncommutative symmetric Markov semigroup. In particular, they showed that if a Riemannian manifold has its Ricci curvature bounded below by $\lambda$, then the heat semigroup $T_t$ satisfies
\begin{align}\label{eq:gamma2}
\left[\Gamma(T_t f_i, T_t f_j)\right]_{i,j}\leqslant e^{-2\la t} \left[T_t\Gamma(f_i,f_j)\right]_{i,j}.
\end{align}
Part i) in  Proposition \ref{prop:gradient} extends this to the sub-Laplacian case for a single Markov map $P$. Although \eqref{eq:gamma2} gives a matrix-valued version of Bakry-{\'E}mery's curvature-dimension condition, it is not clear whether \eqref{eq:gamma2} in the matrix-valued case or noncommutative setting implies a MLSI. On the other hand, part ii) of Proposition \ref{prop:gradient} for matrix-valued function $f$
\begin{align}\label{eq:gamma3} \lan \nabla T_t \Bf, A^s (\nabla T_t \Bf)B^{1-s} \ran_{\tr} \leqslant e^{-2\lambda t} \lan\nabla \Bf,  (T_t^\dagger A)^s (\nabla  \Bf)(T_t^\dagger B)^{1-s} \ran_{\tr},
\end{align}
is more related to \emph{complete gradient estimates} studied by Wirth and Zhang \cite{WZ}.
Their \emph{complete gradient estimates} were inspired by the \emph{entropic Ricci curvature lower bound} introduced in the work of E.~Carlen and J.~Mass \cite{CM}, and N.~Datta and C.~Rouz\'e in \cite{DR}, which are powerful tools in deriving modified log-Sobolev inequalities for quantum Markov semigroups \cite{BrannanGaoJunge2022, BrannanGaoJunge2021}. It is clear that $\eqref{eq:gamma3}$ is more general than $\eqref{eq:gamma2}$ by choosing $A=B$ and $s=0$, and they coincide when reduced to scalar-valued functions. They differ in matrix-valued cases because the matrix multiplication is not commutative, and there are various ways to multiply a matrix $A$ by $X$, such as $AX$, $XA$ and $A^{s}XA^{1-s}$. Proposition \ref{prop:gradient} basically shows that for sub-Laplacians, \eqref{eq:gamma2} and \eqref{eq:gamma3} are equivalent and both reduce to Bakry-{\'E}mery's curvature dimension condition for scalar-valued functions
\begin{align}\label{eq:gamma4}
\Gamma(T_t f, T_t f)\leqslant e^{-2\la t} T_t\Gamma(f,f)\pl.
\end{align}

One can consider whether the above equivalence holds beyond the sub-Laplacian case. Indeed, the equivalence between \eqref{eq:gamma2} and \eqref{eq:gamma4} holds with for all symmetric Markov semigroup with the same argument as used in Proposition \ref{prop:gradient}. The equivalence between \eqref{eq:gamma2} and \eqref{eq:gamma3} is more involved. To formulate \eqref{eq:gamma3} for a general symmetric generator $P_t=e^{L t}:L^{\infty}(\Omega)\to L^{\infty}(\Omega)$, one needs a densely defined closed derivation operator $\delta:L^{2}(\Omega)\to \mathcal{H}$ such that $L=-\delta^*\delta$, where $\mathcal{H}$ is a Hilbert module or more specifically $L^{2}$-space. It was proved in a preprint by M.~Junge, E.~Ricard and D.~Shlyahktenko \cite{JRS} that every symmetric (quantum) Markov semigroup admits such a derivation on a noncommutative $L^{2}$-space $L^{2}(\M)$ of some von Neumann algebra (see also the more recent preprint \cite{wirth} for the strongly continuous semigroup). The proof for part ii) of Proposition \ref{prop:gradient}  works as long as the range of $\delta$ commutes with the functions in $L^{\infty}(\Omega)$, which is obvious for sub-Laplacians because the range of $\nabla=(X_1,\cdots, X_k)$ can be viewed as the diagonal matrices in $C^\infty(M,\bM_k)$. Nevertheless, in a private communication Melchior Wirth pointed out to us that this is unlikely to hold for quantum Markov semigroups.

Our main Theorem \ref{thm:main} states that any sub-Laplacian $\cL=-\sum_{i=1}^k X_i^*X_i$ generated by a collection of vector fields satisfying (i) H\"ormander's condition
and (ii) the gradient estimates
\begin{align}\Gamma(P_t f,P_tf)\leqslant C(t)P_t\Gamma(f,f)\label{eq:gre}\end{align}
for some integrable function $C(t)$, satisfies a CMLSI. Up to the writing of this paper, such estimates were obtained for $\operatorname{SU}\left( 2 \right)$ in \cite{BaudoinBonnefont2009} and for nilpotent Lie groups in \cite{Melcher2008} including Heisenberg group in \cite{DriverMelcher2005}. An interesting  question is to explore  whether \eqref{eq:gre} holds for a larger class of sub-Laplacians on compact Lie groups, which by our results will imply a CMLSI and a dimension-free estimate for the corresponding quantum Markov semigroups as discussed in Section \ref{sec:su2}. Nevertheless, answering this question is beyond the scope of the current paper.

The intertwining relation
\begin{align}\label{eq:interwine3}
\nabla P_tf=e^{-\lambda t} P_t\nabla f
\end{align}
is obviously a stronger condition than Bakry-{\'E}mery's curvature-dimension inequality
\begin{align}\label{eq:pp=1}|\nabla P_t f|\leqslant C(t) P_t(|\nabla f|)\pl, \pl \forall f\in C_c^\infty(M), t>0, \end{align}
hence can be used to derive an CMLSI for the semigroup $P_t$ itself. Indeed, the Ornstein–Uhlenbeck semigroup on $\mathbb{R}^n$ satisfies $\nabla P_t f= e^{- t}P_t\nabla f$ with $\lambda=1$, which is an alternative approach to a CMLSI for the Gaussian measure in \cite{JLLR}.  In the sub-Riemannian setting, \eqref{eq:interwine3} is never satisfied because of lack of curvature lower bound. So our method here does not applies to Ornstein–Uhlenbeck semigroup on Heisenberg group studied in \cite{li2006estimation,BakryBaudoinBonnefontChafai2008}, for which the standard (scalar-valued) log-Sobolev inequality can be obtained from \eqref{eq:pp=1}. In that sense, finding a
 CMLSI for the sub-Laplacian Ornstein–Uhlenbeck semigroup on the Heisenberg group remains an open question. For the more applications of intertwining relations to CMLSIs in the noncommutative setting, see \cite{CM18,CM,BrannanGaoJunge2022,BrannanGaoJunge2021,WZ,wirth2021curvature,DR,JLLR}.

\bibliography{fdiv}

\providecommand{\bysame}{\leavevmode\hbox to3em{\hrulefill}\thinspace}
\providecommand{\MR}{\relax\ifhmode\unskip\space\fi MR }
\providecommand{\MRhref}[2]{%
  \href{http://www.ams.org/mathscinet-getitem?mr=#1}{#2}
}
\providecommand{\href}[2]{#2}
\begin{thebibliography}{10}

\bibitem{BakryEmery1985}
D.~Bakry and Michel {\'E}mery, \emph{Diffusions hypercontractives}, S\'eminaire
  de probabilit\'es, {XIX}, 1983/84, Lecture Notes in Math., vol. 1123,
  Springer, Berlin, 1985, pp.~177--206. \MR{889476 (88j:60131)}

\bibitem{BakryBaudoinBonnefontChafai2008}
Dominique Bakry, Fabrice Baudoin, Michel Bonnefont, and Djalil Chafai, \emph{On
  gradient bounds for the heat kernel on the {H}eisenberg group}, J. Funct.
  Anal. \textbf{255} (2008), no.~8, 1905--1938. \MR{2462581 (2010m:35534)}

\bibitem{BakryGentilLedouxBook}
Dominique Bakry, Ivan Gentil, and Michel Ledoux, \emph{Analysis and geometry of
  {M}arkov diffusion operators}, Grundlehren der Mathematischen Wissenschaften
  [Fundamental Principles of Mathematical Sciences], vol. 348, Springer, Cham,
  2014. \MR{3155209}

\bibitem{BardetCapelGaoLuciaPerez-GarciaRouze2021}
Ivan Bardet, \'{A}ngela Capel, Li~Gao, Angelo Lucia, David P\'{e}rez-Garcia,
  and Cambyse Rouz\'{e}, \emph{Entropy decay for {D}avies semigroups of a one
  dimensional quantum lattice}, arXiv preprint arXiv:2112.00601, 2021.

\bibitem{BaudoinBonnefont2009}
Fabrice Baudoin and Michel Bonnefont, \emph{The subelliptic heat kernel on
  {${\rm SU}(2)$}: representations, asymptotics and gradient bounds}, Math. Z.
  \textbf{263} (2009), no.~3, 647--672. \MR{2545862 (2011d:58060)}

\bibitem{BaudoinBonnefontGarofalo2014}
Fabrice Baudoin, Michel Bonnefont, and Nicola Garofalo, \emph{A
  sub-{R}iemannian curvature-dimension inequality, volume doubling property and
  the {P}oincar\'e inequality}, Math. Ann. \textbf{358} (2014), no.~3-4,
  833--860. \MR{3175142}

\bibitem{BaudoinGarofalo2017}
Fabrice Baudoin and Nicola Garofalo, \emph{Curvature-dimension inequalities and
  {R}icci lower bounds for sub-{R}iemannian manifolds with transverse
  symmetries}, J. Eur. Math. Soc. (JEMS) \textbf{19} (2017), no.~1, 151--219.
  \MR{3584561}

\bibitem{BaudoinKimWang2016}
Fabrice Baudoin, Bumsik Kim, and Jing Wang, \emph{Transverse {W}eitzenb\"ock
  formulas and curvature dimension inequalities on {R}iemannian foliations with
  totally geodesic leaves}, Comm. Anal. Geom. \textbf{24} (2016), no.~5,
  913--937. \MR{3622309}

\bibitem{beigi2020quantum}
Salman Beigi, Nilanjana Datta, and Cambyse Rouz{\'e}, \emph{Quantum reverse
  hypercontractivity: its tensorization and application to strong converses},
  Communications in Mathematical Physics \textbf{376} (2020), no.~2, 753--794.

\bibitem{BrannanGaoJunge2021}
Michael Brannan, Li~Gao, and Marius Junge, \emph{Complete logarithmic sobolev
  inequality via ricci curvature bounded below ii}, Journal of Topology and
  Analysis (2021), 1--54.

\bibitem{BrannanGaoJunge2022}
\bysame, \emph{Complete logarithmic {S}obolev inequalities via {R}icci
  curvature bounded below}, Adv. Math. \textbf{394} (2022), Paper No. 108129,
  60. \MR{4348697}

\bibitem{CRF}
{\'A}ngela Capel, Cambyse Rouz{\'e}, and Daniel~Stilck Fran{\c{c}}a, \emph{The
  modified logarithmic {S}obolev inequality for quantum spin systems: classical
  and commuting nearest neighbour interactions}, arXiv preprint
  arXiv:2009.11817, 2020.

\bibitem{CarboneSasso2008}
Raffaella Carbone and Emanuela Sasso, \emph{Hypercontractivity for a quantum
  {O}rnstein-{U}hlenbeck semigroup}, Probab. Theory Related Fields \textbf{140}
  (2008), no.~3-4, 505--522. \MR{2365482}

\bibitem{CM}
Eric~A Carlen and Jan Maas, \emph{Gradient flow and entropy inequalities for
  quantum markov semigroups with detailed balance}, Journal of Functional
  Analysis \textbf{273} (2017), no.~5, 1810--1869.

\bibitem{CM18}
\bysame, \emph{Non-commutative calculus, optimal transport and functional
  inequalities in dissipative quantum systems}, Journal of Statistical Physics
  \textbf{178} (2020), no.~2, 319--378.

\bibitem{ChruscinskiPascazio2017}
Dariusz Chru\'{s}ci\'{n}ski and Saverio Pascazio, \emph{A brief history of the
  {GKLS} equation}, Open Syst. Inf. Dyn. \textbf{24} (2017), no.~3, 1740001,
  20. \MR{3713631}

\bibitem{DR}
Nilanjana Datta and Cambyse Rouz{\'e}, \emph{Relating relative entropy, optimal
  transport and fisher information: A quantum hwi inequality}, Annales Henri
  Poincar{\'e}, Springer, 2020, pp.~1--36.

\bibitem{diaconis1996logarithmic}
Persi Diaconis and Laurent Saloff-Coste, \emph{Logarithmic sobolev inequalities
  for finite markov chains}, The Annals of Applied Probability \textbf{6}
  (1996), no.~3, 695--750.

\bibitem{DriverMelcher2005}
Bruce~K. Driver and Tai Melcher, \emph{Hypoelliptic heat kernel inequalities on
  the {H}eisenberg group}, J. Funct. Anal. \textbf{221} (2005), 340--365.

\bibitem{gao2020fisher}
Li~Gao, Marius Junge, and Nicholas LaRacuente, \emph{{F}isher information and
  logarithmic {S}obolev inequality for matrix-valued functions}, Annales Henri
  Poincar{\'e}, vol.~21, Springer, 2020, pp.~3409--3478.

\bibitem{GaoRouze2021}
Li~Gao and Cambyse Rouz{\'e}, \emph{{C}omplete entropic inequalities for
  quantum {M}arkov chains}, arXiv preprint arXiv:2102.04146, 2021.

\bibitem{GordinaLaetsch2016}
Maria Gordina and Thomas Laetsch, \emph{Sub-{L}aplacians on {s}ub-{R}iemannian
  {M}anifolds}, Potential Anal. \textbf{44} (2016), no.~4, 811--837.
  \MR{3490551}

\bibitem{GordinaLaetsch2017}
\bysame, \emph{A convergence to {B}rownian motion on sub-{R}iemannian
  manifolds}, Trans. Amer. Math. Soc. \textbf{369} (2017), no.~9, 6263--6278,
  In print: September 2017. \MR{3660220}

\bibitem{GoriniKossakowskiSudarshan1976}
Vittorio Gorini, Andrzej Kossakowski, and E.~C.~G. Sudarshan, \emph{Completely
  positive dynamical semigroups of {$N$}-level systems}, J. Mathematical Phys.
  \textbf{17} (1976), no.~5, 821--825. \MR{406206}

\bibitem{Gross1975a}
Leonard Gross, \emph{Hypercontractivity and logarithmic {S}obolev inequalities
  for the {C}lifford {D}irichlet form}, Duke Math. J. \textbf{42} (1975),
  no.~3, 383--396. \MR{0372613 (51 \#8820)}

\bibitem{Gross1975c}
\bysame, \emph{Logarithmic {S}obolev inequalities}, Amer. J. Math. \textbf{97}
  (1975), no.~4, 1061--1083. \MR{MR0420249 (54 \#8263)}

\bibitem{Gross14}
\bysame, \emph{Hypercontractivity, logarithmic sobolev inequalities, and
  applications: a survey of surveys}, Diffusion, quantum theory, and radically
  elementary mathematics \textbf{47} (2014), 45--73.

\bibitem{Hormander1967a}
Lars H{\"o}rmander, \emph{Hypoelliptic second order differential equations},
  Acta Math. \textbf{119} (1967), 147--171. \MR{0222474 (36 \#5526)}

\bibitem{JLLR}
Marius Junge, Haojian Li, and Nicholas LaRacuente, \emph{Graph {H}\"{o}rmander
  systems}, arXiv preprint arXiv:2006.14578, 2020.

\bibitem{JRS}
Marius Junge, Eric Ricard, and Dimitri Shlyahktenko, \emph{Noncommutative
  {d}iffusion {s}emigroups and {f}ree {p}robability}, unpublished notes, 2013.

\bibitem{JZ}
Marius Junge and Qiang Zeng, \emph{Noncommutative martingale deviation and
  poincar{\'e} type inequalities with applications}, Probability Theory and
  Related Fields \textbf{161} (2015), no.~3-4, 449--507.

\bibitem{KastoryanoTemme2013}
Michael~J. Kastoryano and Kristan Temme, \emph{Quantum logarithmic {S}obolev
  inequalities and rapid mixing}, J. Math. Phys. \textbf{54} (2013), no.~5,
  052202, 30. \MR{3098923}

\bibitem{king2014hypercontractivity}
Christopher King, \emph{Hypercontractivity for semigroups of unital qubit
  channels}, Communications in Mathematical Physics \textbf{328} (2014), no.~1,
  285--301.

\bibitem{Ledoux}
Michel Ledoux, \emph{Spectral gap, logarithmic sobolev constant, and geometric
  bounds}, Surveys in differential geometry \textbf{9} (2004), no.~1, 219--240.

\bibitem{li2006estimation}
Hong-Quan Li, \emph{Estimation optimale du gradient du semi-groupe de la
  chaleur sur le groupe de heisenberg}, Journal of Functional Analysis
  \textbf{236} (2006), no.~2, 369--394.

\bibitem{Lieb1973}
Elliott~H. Lieb, \emph{Convex trace functions and the {W}igner-{Y}anase-{D}yson
  conjecture}, Advances in Math. \textbf{11} (1973), 267--288. \MR{332080}

\bibitem{LindbladG1976}
G.~Lindblad, \emph{On the generators of quantum dynamical semigroups}, Comm.
  Math. Phys. \textbf{48} (1976), no.~2, 119--130. \MR{413878}

\bibitem{LugiewiczZegarlinski2007}
P.~{\L}ugiewicz and B.~Zegarli\'{n}ski, \emph{Coercive inequalities for
  {H}\"{o}rmander type generators in infinite dimensions}, J. Funct. Anal.
  \textbf{247} (2007), no.~2, 438--476. \MR{2323442}

\bibitem{Melcher2008}
Tai Melcher, \emph{Hypoelliptic heat kernel inequalities on {L}ie groups},
  Stochastic Process. Appl. \textbf{118} (2008), no.~3, 368--388. \MR{2389050
  (2009a:58053)}

\bibitem{MontgomeryBook2002}
Richard Montgomery, \emph{A tour of subriemannian geometries, their geodesics
  and applications}, Mathematical Surveys and Monographs, vol.~91, American
  Mathematical Society, Providence, RI, 2002. \MR{1867362 (2002m:53045)}

\bibitem{OlkiewiczZegarlinski1999}
Robert Olkiewicz and Bogus{\l}aw Zegarlinski, \emph{Hypercontractivity in
  noncommutative {$L_p$} spaces}, J. Funct. Anal. \textbf{161} (1999), no.~1,
  246--285. \MR{1670230 (2000e:46079)}

\bibitem{petz1996monotone}
D{\'e}nes Petz, \emph{Monotone metrics on matrix spaces}, Linear algebra and
  its applications \textbf{244} (1996), 81--96.

\bibitem{Rothaus1985}
O.~S. Rothaus, \emph{Analytic inequalities, isoperimetric inequalities and
  logarithmic {S}obolev inequalities}, J. Funct. Anal. \textbf{64} (1985),
  no.~2, 296--313. \MR{812396}

\bibitem{Strichartz1986a}
Robert~S. Strichartz, \emph{Sub-{R}iemannian geometry}, J. Differential Geom.
  \textbf{24} (1986), no.~2, 221--263. \MR{862049 (88b:53055)}

\bibitem{Strichartz1986aCorrections}
\bysame, \emph{Corrections to: ``{S}ub-{R}iemannian geometry'' [{J}.
  {D}ifferential {G}eom.\ {\bf 24} (1986), no.\ 2, 221--263; {MR}0862049
  (88b:53055)]}, J. Differential Geom. \textbf{30} (1989), no.~2, 595--596.
  \MR{1010174 (90f:53081)}

\bibitem{VSCC}
Nicholas~T Varopoulos, Laurent Saloff-Coste, and Thierry Coulhon,
  \emph{Analysis and geometry on groups}, vol. 100, Cambridge university press,
  2008.

\bibitem{Verdu2014}
Sergio Verdu, \emph{Total variation distance and the distribution of relative
  information}, 2014 Information Theory and Applications Workshop (ITA), 2014,
  pp.~1--3.

\bibitem{wirth}
Melchior Wirth, \emph{A noncommutative transport metric and symmetric quantum
  markov semigroups as gradient flows of the entropy}, arXiv preprint
  arXiv:1808.05419 (2018).

\bibitem{WZ}
Melchior Wirth and Haonan Zhang, \emph{Complete gradient estimates of quantum
  markov semigroups}, arXiv preprint, 2020.

\bibitem{wirth2021curvature}
\bysame, \emph{Curvature-dimension conditions for symmetric quantum markov
  semigroups}, arXiv preprint arXiv:2105.08303 (2021).

\end{thebibliography}
\bibliographystyle{amsplain}

\end{document}